\documentclass[a4paper, 11pt]{amsart}
\usepackage{graphicx}
\usepackage[all]{xy}

\newcommand{\ko}{\;\; ,}

\numberwithin{equation}{subsection}
\newtheorem{theorem}[subsection]{Theorem}
\newtheorem{classification-theorem}[subsection]{Classification Theorem}
\newtheorem{decomposition-theorem}[subsection]{Decomposition Theorem}
\newtheorem{proposition-definition}[subsection]{Proposition-Definition}
\newtheorem{periodicity-conjecture}[subsection]{Periodicity Conjecture}
\newtheorem{lemma}[subsection]{Lemma}
\newtheorem{proposition}[subsection]{Proposition}
\newtheorem{corollary}[subsection]{Corollary}

\newtheorem{assumption}[subsection]{Assumption}

\newtheorem{remark}[subsection]{Remark}

\newcommand{\rad}{\operatorname{rad}\nolimits}
\newcommand{\res}{\operatorname{res}\nolimits}
\newcommand{\can}{\operatorname{can}\nolimits}

\newcommand{\Gr}{\operatorname{{Gr}}\nolimits}

\newcommand{\Hom}{\operatorname{Hom}\nolimits}
\newcommand{\im}{\operatorname{im}\nolimits}
\newcommand{\cok}{\operatorname{cok}\nolimits}

\newcommand{\Ext}{\operatorname{Ext}\nolimits}

\newcommand{\Mod}{\operatorname{Mod }\nolimits}
\renewcommand{\mod}{\operatorname{mod}\nolimits}
\newcommand{\grm}{\operatorname{grm}\nolimits}
\newcommand{\rep}{\operatorname{rep}\nolimits}
\newcommand{\proj}{\operatorname{proj}\nolimits}

\newcommand{\ind}{\operatorname{ind}\nolimits}

\newcommand{\Z}{\operatorname{\mathbb{Z}}\nolimits}

\newcommand{\N}{\operatorname{\mathbb{N}}\nolimits}

\newcommand{\dimv}{\operatorname{\underline{dim}}\nolimits}

\newcommand{\GL}{\operatorname{Gl}\nolimits}
\newcommand{\ten}{\otimes}

\newcommand{\bt}{\bullet}

\newcommand{\iso}{\stackrel{_\sim}{\rightarrow}}

\newcommand{\id}{\mathbf{1}}

\newcommand{\ca}{{\mathcal A}}
\newcommand{\cb}{{\mathcal B}}
\newcommand{\cc}{{\mathcal C}}
\newcommand{\cd}{{\mathcal D}}
\newcommand{\ce}{{\mathcal E}}

\newcommand{\ch}{{\mathcal H}}

\newcommand{\cm}{{\mathcal M}}

\newcommand{\cp}{{\mathcal P}}
\newcommand{\cR}{{\mathcal R}}

\newcommand{\cs}{{\mathcal S}}

\newcommand{\cv}{{\mathcal V}}

\newcommand{\eps}{\varepsilon}
\renewcommand{\tilde}[1]{\widetilde{#1}}
\newcommand{\ul}[1]{\underline{#1}}

\renewcommand{\hat}[1]{\widehat{#1}}
\begin{document}
\baselineskip=15pt
\title[Desingularizations via quiver varieties]{Desingularizations of quiver Grassmannians\\
via graded quiver varieties}
\author{Bernhard Keller and Sarah Scherotzke}
\address{B.~K.~: Universit\'e Paris Diderot -- Paris~7, Institut
Universitaire de France, UFR de Math\'ematiques, Institut de
Math\'ematiques de Jussieu, UMR 7586 du CNRS, Case 7012, B\^atiment
Sophie Germain, 75205 Paris Cedex 13, France}
\address{S.~S.~: University of Bonn, Mathematisches Institut, 
Endenicher Allee 60, 53115 Bonn, Germany}

\email{keller@math.jussieu.fr, sarah@math.uni-bonn.de}


\begin{abstract} 
Inspired by recent work of Cerulli--Feigin--Reineke on desingularizations
of quiver Grassmannians of representations of Dynkin quivers, we
obtain desingularizations in considerably more general situations and in
particular for Grassmannians of modules over iterated tilted algebras of Dynkin type.  
Our desingularization map is constructed from Nakajima's desingularization map 
for graded quiver varieties. 
\end{abstract}

\maketitle

\tableofcontents

\section{Introduction and main results}
\label{s:intro}
A quiver Grassmannian is the variety of subrepresentations with
given dimension vector of a fixed quiver representation. To the best
of the authors' knowledge, quiver Grassmannians first appeared in 
Schofield's work \cite{Schofield92a}. They are projective varieties and 
Reineke shows in \cite{Reineke12} that every projective variety can be 
realized as a quiver Grassmannian 
(we refer to the final example of Hille's \cite{Hille96} for a similar, and in fact
closely related result, and to Ringel's \cite{Ringel13} for an analogous
`universality theorem' in the setting of Auslander algebras). 
Caldero-Chapoton discovered
\cite{CalderoChapoton06} that the canonical generators
of Fomin-Zelevinsky's  cluster algebras \cite{FominZelevinsky02}
can be interpreted as generating polynomials of Euler characteristics
of quiver Grassmannians. Since then, quiver Grassmannians have 
played an important role in the additive categorification of 
(quantum) cluster algebras,
cf. for example \cite{CalderoChapoton06} \cite{CalderoKeller08}
\cite{DerksenWeymanZelevinsky10} \cite{Nakajima11} \cite{Qin12} \cite{Efimov11}. 

In \cite{CerulliFeiginReineke12}, \cite{CerulliFeiginReineke12b},
Cerulli--Feigin--Reineke initiated a systematic study of 
(singular) quiver Grassmannians of Dynkin quivers, starting from the 
surprising observation that the type $A$
degenerate flag varieties studied in \cite{Feigin11} \cite{Feigin12}
\cite{FeiginFinkelberg11} are of this form.
An important aspect of their work is the construction of
desingularizations, which they achieve in their recent
paper \cite{CerulliFeiginReineke12a}
generalizing \cite{FeiginFinkelberg11}.
In \cite{CerulliFeiginReineke13}, they link these desingularizations to a construction
by Hernandez--Leclerc \cite{HernandezLeclerc10}, which
has been generalized by Leclerc--Plamondon \cite{LeclercPlamondon12}
and further generalized by the present authors in \cite{KellerScherotzke13a}.

In this article, we build on \cite{KellerScherotzke13a}  
to construct desingularizations of quiver Grassmannians 
in much more general situations and in particular for all modules 
over the repetitive algebra of  an arbitrary iterated tilted algebra $B$ of Dynkin
type, like the algebra $B$ given by the square quiver
\[
\xymatrix{
1 \ar[d] \ar[r] & 2 \ar[d] \\
3 \ar[r] & 4
}
\]
with the commutativity relation ($B$ is tilted of type $D_4$).
The main ingredient of our construction is the desingularization map for 
graded quiver varieties introduced by Nakajima \cite{Nakajima01} \cite{Nakajima11}
and generalized from bipartite to acylic quivers by Qin \cite{Qin12}
\cite{Qin12a}.

More precisely, we consider a module $M$ over the singular
Nakajima category $\cs$ associated with an acyclic quiver $Q$,
cf. section~\ref{ss:the-Nakajima-categories}.
By \cite{LeclercPlamondon12}, such
a module corresponds to a point in the graded affine quiver
variety $\cm_0(d)$ associated with $Q$ and the dimension
vector $d=\dimv M$ of $M$. Nakajima has constructed a pre-desingularization
(i.e. a proper, surjective morphism with smooth domain)
\[
\pi: \cm(d) \to \cm_0(d)
\] 
of $\cm_0(d)$. Here the points of $\cm(d)$ can be interpreted as 
(orbits of stable) representations of the regular Nakajima category $\cR$,
which contains $\cs$ as a full subcategory, and the 
map $\pi$ takes a representation $L$ of 
$\cR$ to its restriction $\res(L)$ to $\cs\subset\cR$. 
We will show that for suitable modules $M$,
each quiver Grassmannian of $M$ admits a desingularization
by a disjoint union of connected components of quiver Grassmannians 
of a distinguished point in the fiber of $\pi$ over $M$, namely the so-called
intermediate Kan extension $K_{LR}(M)$ of
section~2.10 of \cite{KellerScherotzke13a}.

Let us describe our main results more precisely. 
Let $M$ be a finite-dimensional $\cs$-module
of dimension vector $d$.
Let $w$ be a dimension vector less or equal
to $d$. Using Nakajima's stratification
of $\cm_0(d)$, we assign a dimension vector
$(v_C,w)$ of $\cR$ with each irreducible component
$C$ of the quiver Grassmannian $\Gr_{w}(M)$, cf.~Lemma~\ref{lemma:definition-vc}.
Let $\cv_w(M)$ be the set of the vectors $v_C$.
Recall that a module is {\em rigid} if its space
of selfextensions vanishes. The following result
is modeled on Theorem~7.4 of \cite{CerulliFeiginReineke12a}
with the intermediate Kan extension $K_{LR}(M)$ playing
the role of the module $\hat{M}$ of [loc. cit.].

\begin{theorem}[Theorem~\ref{thm:desingularization}]
\label{thm:intro-pre-desingularization}
Suppose that $K_{LR}(M)$ is rigid. Then the map
\[
\pi_{\Gr}: \coprod_{v\in \cv_w(M)}\Gr_{(v,w)}(K_{LR}(M)) \to 
\Gr_{w}(M)
\]
taking $U\subset K_{LR}(M)$ to $\res(U) \subset M$ is a pre-desingularization
(a proper, surjective morphism with smooth domain). 
\end{theorem}

We determine the fibres of the map $\pi_{\Gr}$ in Theorem \ref{thm:fibre-of-piGr}. 
To make sure that the generic fibre is reduced to a point, we need
to shrink the domain of $\pi_{\Gr}$. We do this as follows:
An $\cR$-module is {\em bistable} if it is isomorphic
to the intermediate extension of some $\cs$-module. For a dimension
vector $(v,w)$ of $\cR$, denote by $\Gr^{bs}_{(v,w)}(K_{LR}(M))$ 
the {\em bistable Grassmannian}, i.e. the closure of the set of points 
corresponding to bistable submodules. 
In analogy with Remark~7.8 of [loc. cit.], we conjecture
that the bistable Grassmannian actually equals the whole Grassmannian.
The following result is modeled on Corollary~7.7
of \cite{CerulliFeiginReineke12a} with the bistable Grassmannians
playing the role of the sets $\overline{S_{[N]}}$ in [loc. cit.]. 

\begin{theorem}[Theorem~\ref{thm:desingularization}]
\label{thm:intro-desingularization} 
Suppose that $K_{LR}(M)$ is rigid.
The map
\[
\pi^{bs}: \coprod_{v\in \cv_w(M)}\Gr^{bs}_{(v,w)}(K_{LR}(M)) \to 
\Gr_{w}(M)
\]
taking $U\subset K_{LR}(M)$ to $\res(U) \subset M$ is a desingularization
(a proper, surjective morphism with smooth domain which induces
an isomorphism between dense open subsets).
\end{theorem}

We will give sufficient conditions for $K_{LR}(M)$ to be rigid
(Lemmas~\ref{lemma:necessary-condition-rigidity} and \ref{lemma:Auslander-category})
and show by an example that this is not always the case 
(section~\ref{ss:non-rigid-inter-ext}). Nevertheless,
as a consequence of the above theorems, we will
obtain desingularizations for all modules
over the repetitive algebra of an iterated tilted algebra of
Dynkin type (Corollary~\ref{cor:desing-quiver-grassmannian}).
We will show that this covers in particular all the cases considered in 
\cite{CerulliFeiginReineke12a} and yields a natural interpretation
for the algebra $\ch_Q$ of [loc. cit.]
(cf. section~\ref{ss:link-to-CFRs-desingularization}).


The paper is organized as follows. In section~\ref{s:intermediate-extensions}, 
we introduce the intermediate extension $F_{\lambda \rho}$ 
associated to a localization functor between abelian categories
$F:\ca \to\cb$ which admits a right and a left adjoint.
In Lemma~\ref{lemma:necessary-condition-rigidity}, 
we give sufficient conditions for an object in the image of 
$F_{ \lambda \rho}$ to be rigid. In section~\ref{ss:Auslander-category},
we examine the particular case where $F$ is the 
restriction 
\[
\mod(\mod(\cp)) \to \mod(\cp)
\]
along the Yoneda embedding $\cp \to \mod(\cp)$ from a coherent
category $\cp$ to its category of finitely presented modules.

In section~\ref{s:desingularization}, we recall the definition of the regular and the singular 
Nakajima categories $\cR$ and $\cs$ introduced in \cite{LeclercPlamondon12}
(cf.~also section~2 of \cite{KellerScherotzke13a}) and explain how they relate to 
Nakajima's graded quiver varieties. In section~\ref{ss:configurations}, for
certain subsets $\sigma^{-1}(C)$ of the set of vertices of $\cs$, we
consider the quotients $\cs_C$ and $\cR_C$ obtained by factoring
out the ideal generated by the identities of the vertices outside $\sigma^{-1}(C)$.
We will later need the extra generality afforded by a suitable choice of $C$
in order to recover the results of Cerulli--Feigin--Reineke \cite{CerulliFeiginReineke12a}.
In section~~\ref{ss:proof-of-desingularization},  we state and prove our main 
Theorems.

In section~\ref{ss:piecewise-hereditary}, we show that
for a suitable choice of $C$, the category $\cs_C$ is equivalent to the category
of indecomposable projectives over the repetitive algebra $\hat{A}$ 
of an iterated tilted algebra $A$ of Dynkin type and that the category 
$\cR_C$ is equivalent to the category of indecomposable representations of 
$\hat{A}$ (Proposition~\ref{prop:mesh-category-isomorphisms}).
By Lemma~\ref{lemma:Auslander-category}, we know that the intermediate 
extension of a finite-dimensional $\hat{A}$-module is always rigid. Thus, we obtain 
a desingularization map for any finite-dimensional module over the 
repetitive algebra $\hat{A}$ of $A$. We then specialize $A$ to the
path algebra $kQ$ of a Dynkin quiver $Q$. The category of
finite-dimensional $kQ$-modules appears naturally as a full subcategory of 
$\mod(\hat{kQ})$. In section~\ref{ss:link-to-CFRs-desingularization}, 
we show that the intermediate extension $K_{LR}$
restricted to the category of finite-dimensional $kQ$-modules specializes
to the functor $\Lambda$ constructed by Cerulli--Feigin--Reineke 
\cite{CerulliFeiginReineke12a} and that our desingularization specializes 
to theirs. In section~\ref{s:example-tilted-D4}, we illustrate the desingularization
theorem using a module over a tilted algebra of type~$D_4$.

\subsection*{Acknowledgments} 
This article was conceived during the  cluster algebra program at the 
MSRI in fall 2012. The authors are  grateful to the MSRI for financial support 
and ideal working conditions.
They thank Giovanni Cerulli Irelli, Bernard Leclerc,
Pierre-Guy Plamondon and Markus Reineke for stimulating conversations.

\section{Intermediate Kan extensions}
\label{s:intermediate-extensions}

\subsection{Intermediate Kan extensions and rigidity}
\label{ss:rigidity-of-intermediate-extensions} We first study the
properties of the intermediate extension in the framework of abelian
categories: Let $\ca$ and $\cb$ be abelian categories. 
Let $F: \ca \to \cb$ be a
localization functor, i.e. $F$ is exact and induces an
equivalence 
\[
\ca/\ker(F) \iso \cb \ko
\]
where $\ca/\ker(F)$ is the localization of $\ca$ with respect
to the Serre subcategory $\ker(F)$ in the sense of \cite{Gabriel62}.
We assume that $F$ admits both a right adjoint $F_\rho$ and a
left adjoint $F_\lambda$ so that we have three adjoint functors
\[
\xymatrix{
\ca \ar[d]|-*+{{\scriptstyle F}} \\
\cb. \ar@<2ex>[u]^{F_\lambda} \ar@<-2ex>[u]_{F_\rho}
}
\]
Notice that $F_\lambda$ and $F_\rho$ are both fully faithful
(since $F$ is a localization), that $F_\lambda$ is right exact
and preserves projectivity and that $F_\rho$ is left exact
and preserves injectivity. Denote the adjunction morphisms
by
\[
\phi: F_\lambda F \to \id_\ca \ko
\psi: \id_\cb \to F F_\lambda \ko
\eta: F F_\rho \to \id_\cb \ko
\eps: \id_\ca \to F_\rho F.
\]

\begin{lemma} \label{lemma:characterization-of-stables} Let $M$ be an object of $\ca$.
\begin{itemize} 
\item[a)] The adjunction morphism $M \to F_\rho F M$
is mono if and only if the group $\Hom(N,M)$ vanishes for each object $N$ of $\ker(F)$.
\item[b)] The adjunction morphism $M\to F_\rho FM$ is invertible
iff we have $\Hom(N,M)=0=\Ext^1(N,M)$ for each object $N$ of $\ker(F)$.
\end{itemize}
\end{lemma}

We leave the proof of part a) as an exercise for the reader. Part b) is
the characterization of the image of the adjoint
of a localization functor given in Lemme~1, page~370 of \cite{Gabriel62}.
We call an object $M$ {\em stable} if it satisfies the conditions part a). Dually,
it is {\em co-stable} if it satisfies the dual conditions: the adjunction morphism 
$F_\lambda F M \to M$ is epi or, equivalently, we have $\Hom(M,N)=0$
for each object $N$ of $\ker(F)$. 

\begin{lemma} \label{lemma:commutative-square} 
The following square is commutative
\[
\xymatrix{
F_\lambda F F_\rho \ar[r]^-\sim_-{F_\lambda \eta} \ar[d]_{\phi F_\rho} & 
F_\lambda \ar[d]^{\eps F_\lambda}  \\
F_\rho \ar[r]^-\sim_-{F_\rho \psi} & F_\rho F F_\lambda
}
\]
\end{lemma}

\begin{proof} Since $F:\ca\to\cb$ is essentially surjective, it suffices
to check the commutativity after pre-composing with $F$. Consider
the diagram
\[
\xymatrix{
F_\lambda F \ar[r]^-\sim_-{F_\lambda F \eps} \ar[d]_\phi &
F_\lambda F F_\rho F \ar[r]^-\sim_{F_\lambda\eta F} \ar[d]_{\phi F_\rho F} & 
F_\lambda F \ar[r]^\phi \ar[d]^{\eps F_\lambda F} &
\id_\ca \ar[d]^\eps \\
\id_\ca  \ar[r]_-\eps & 
F_\rho F \ar[r]^\sim_-{F_\rho \psi F} &
F_\rho F F_\lambda F \ar[r]^-\sim_-{F_\rho F \phi} &
F_\rho F.
}
\]
Here the composition $(F_\lambda \eta F)(F_\lambda F \eps)$ is the
identity and so is the composition $(F_\rho F \phi)(F_\rho \psi F)$. Thus, the
large rectangle is commutative. The leftmost square is commutative because
$\phi: F_\lambda F \to \id_\ca$ is a natural transformation and the rightmost
square is commutative because so is $\eps: \id_\ca \to F_\rho F$. It follows
that the central square is commutative as claimed.
\end{proof}

By Lemma~\ref{lemma:commutative-square}, we have a canonical morphism
\[
\can : F_\lambda \to F_\rho.
\]
We define the {\em intermediate extension $F_{\lambda \rho}$} to be its image.
Notice that, for each object $M$ of $\cb$, the object $F_{\lambda\rho}(M)$
is both stable (as a subobject of $F_\rho M$) and co-stable (as a quotient
of $F_\lambda M$). We have canonical morphisms
\[
\xymatrix{F_\lambda \ar[r]^\pi & F_{\lambda \rho} \ar[r]^\iota & F_\rho}
\]
and their images under $F$ are invertible (since $F \iota$ is mono,
$F\pi$ is epi and their composition $F(\can)$ is invertible).
One deduces that $F_{\lambda\rho}$ induces an equivalence
from $\cb$ onto the full subcategory of $\ca$ formed by the
objects which are both stable and co-stable.

\begin{lemma} \label{lemma:necessary-condition-rigidity}
\begin{itemize} 
\item[a)] For all objects $L$ of $\ca$ and $M$ of $\cb$, we have
canonical injections
\[
\Ext^1_\ca(L, F_\rho M) \to \Ext^1_\cb(FL, M) \mbox{ and }
\Ext^1_\ca(F_\lambda M, L) \to \Ext^1_\cb(M, FL).
\]
\item[b)] If $L$ is rigid in $\cb$, then $F_\lambda(L)$, $F_\rho(L)$ and
$F_{\lambda\rho}(L)$ are rigid in $\ca$.
\item[c)] Conversely, if $F_{\lambda\rho}(L)$ is rigid in $\ca$ and 
moreover the group $\Ext^2_\ca(U,U)$ vanishes, where $U$ is the cokernel of 
$F_{\lambda\rho}(L) \to F_\rho(L)$, then $L$ is rigid in $\cb$.
\end{itemize}
\end{lemma}

\begin{proof} a) We define the image of the class of an exact sequence
\[
\xymatrix{ 0 \ar[r] & F_\rho M \ar[r]^-i & E \ar[r]^p & L \ar[r] & 0}
\]
to be the class of the sequence
\[
\xymatrix{ 0 \ar[r] & M \ar[r]^j & FE \ar[r]^{Fp} & FL \ar[r] & 0} \ko
\]
where $j=(Fi)(\eta M)^{-1}$.  Clearly, this yields a well-defined map
\[
\Ext^1_\ca(L, F_\rho M) \to \Ext^1_\cb(L,M).
\]
It is not hard to check that if $r: FE\to M$ is a retraction for $j$, then
$(F_\rho r)(\eta E)$ is a retraction for $i$. Thus, our map is injective. Dually,
one obtains the second injection.

b) By part a), we have the injection
\[
\Ext^1_\ca(F_\lambda L, F_\lambda L) \to \Ext^1_\cb(L, F F_\lambda L) =
\Ext^1_\cb(L,L)=0
\]
and similarly for $F_\rho L$. Now consider the exact sequence
\[
\xymatrix{ 0 \ar[r] & F_{\lambda \rho}(L) \ar[r] & F_\rho(L) \ar[r] & U \ar[r] & 0.}
\]
Here the object $U$ lies in $\ker(F)$. If we apply the functor 
$\Hom_\ca(F_{\lambda\rho}(L), ?)$ to the exact sequence, we obtain 
the exact sequence
\[
\Hom_\ca(F_{\lambda\rho}(L),U) \to 
\Ext^1_\ca(F_{\lambda\rho}(L), F_{\lambda\rho}(L)) \to 
\Ext^1_\ca(F_{\lambda\rho}(L), F_\rho(L)).
\]
Since $F_{\lambda\rho}(L)$ is co-stable and $U$ lies in $\ker(F)$, the
left-hand term vanishes. Since we have $F F_{\lambda\rho}(L)=L$ and
$L$ is rigid, the right hand term vanishes by part a). Thus, the
object $F_{\lambda\rho}(L)$ is rigid.

c) Let 
\[
\xymatrix{ 0 \ar[r] & L \ar[r]^i & E \ar[r] & L \ar[r] & 0}
\]
be a non split exact sequence in $\cb$. Then $F_\rho(i)$ is a non
split monomorphism of $\ca$ (since $F_\rho$ is fully faithful) and
so the sequence
\[
\xymatrix{ 0 \ar[r] & F_\rho(L) \ar[r]^{F_\rho(i)} & F_\rho(E) \ar[r] &
\cok(F_\rho(i)) \ar[r] & 0}
\]
is non split in $\ca$. The object $\cok(F_\rho(i))$ is stable since for
$N\in \ker(F)$, we have
\[
\Hom(N, F_\rho(E))=0=\Ext^1_\ca(N,F_\rho(L)).
\]
Moreover, the image $F \cok(F_\rho(i))$  is isomorphic to $L$ since $F$ is
exact. So we have an exact sequence
\[
0 \to F_{\lambda\rho}(L) \to \cok(F_\rho(i)) \to U \to 0
\]
with $U$ in $\ker(F)$. If we apply $\Hom(?, F_\rho(L))$ to this sequence, we
obtain an exact sequence
\[
0 \to \Ext^1(\cok(F_\rho(i)), F_\rho(L)) \to \Ext^1(F_{\lambda\rho}(L), F_\rho(L)).
\]
Since we have found a non zero element in $\Ext^1(\cok(F_\rho(i)), F_\rho(L))$,
we see that the right hand group does not vanish. We claim that
$\Ext^1(F_{\lambda\rho}(L), U)$ vanishes. Indeed, this follows from
our assumption  when we apply $\Ext^1(?, U)$ to the sequence
\[
0 \to F_{\lambda\rho}(L) \to F_\rho(L) \to U \to 0.
\]
Now we claim that we have an isomorphism
\[
\Ext^1(F_{\lambda\rho}(L), F_{\lambda\rho}(L)) \iso
\Ext^1(F_{\lambda\rho}(L), F_\rho(L)).
\]
Indeed this follows by applying $\Ext^1(F_{\lambda\rho}(L),?)$ to the
sequence
\[
0 \to F_{\lambda\rho}(L) \to F_\rho(L) \to U \to 0.
\]
We conclude that $\Ext^1(F_{\lambda\rho}(L), F_{\lambda\rho}(L))$ is
non zero as claimed.
\end{proof}

\subsection{The case of the Auslander category}
\label{ss:Auslander-category} We consider the special case of
the setup of section~\ref{ss:rigidity-of-intermediate-extensions} where
$\cb$ is a module category and $\ca=\mod(\cb)$ the Auslander
category.

Let $k$ be a field and
$\cp$ a skeletally small $k$-category (i.e.~its isomorphism classes
form a set). Let $\mod(\cp)$ be the category of finitely presented
$\cp$-modules, i.e.~of $k$-linear functors $M: \cp^{op} \to \Mod k$
admitting an exact sequence
\[
P_1 \to P_0 \to M \to 0 
\]
where $P_0$ and $P_1$ are finitely generated projective $\cp$-modules,
i.e.~direct factors of finite direct sums of representable $\cp$-modules
$P^{\wedge}=\cp(?,P)$, $P\in\cp$. 
Notice that the category $\mod(\cp)$ is still skeletally small. We
assume that $\mod(\cp)$ is abelian or, equivalently, that $\cp$ is
coherent, i.e. the kernel of any morphism between finitely generated
projective $\cp$-modules is finitely generated.  We have the Yoneda 
embedding $\cp \to \mod(\cp)$ taking 
$P$ to $P^\wedge$.
Let $\cb=\mod(\cp)$, $\ca=\mod(\cb)$ and let $F=\res:\ca\to\cb$ be the
restriction along the Yoneda embedding. Thus, we have functors
\[
\xymatrix{
\mod(\cp) \\
\cp \ar[u]^{\mbox{{\small Yoneda}}}
}
\quad
\xymatrix{
\mod(\mod(\cp)) \ar@{=}[r] \ar[d]|-*+{\res} & \ca \\
\mod(\cp) \ar@{=}[r] \ar@<3ex>[u]^{K_L} \ar@<-3ex>[u]_{K_R} & \cb,
}
\]
where $F_\lambda=K_L$ and $F_\rho=K_R$ are the left and
right Kan extensions adjoint to the restriction functor $F=\res$.
As in section~\ref{ss:rigidity-of-intermediate-extensions},
let $K_{LR}=F_{\lambda\rho}$ be the intermediate extension.

\begin{lemma} \label{lemma:Auslander-category}
\begin{itemize}
\item[a)] The functor $K_R: \ca\to\cb$ is isomorphic to the
Yoneda embedding 
\[
\mod(\cp) \to \mod(\mod(\cp)), \, M \mapsto M^\wedge=\Hom(?,M).
\]
In particular, for each injective $I$ of $\mod(\cp)$, the module
$K_R(I)$ is both projective and injective.
\item[b)] The canonical morphism $K_L(P) \to K_R(P)$ is invertible
for all finitely generated projective $\cp$-modules $P$.
\item[c)] Let $M$ be in $\mod(\cp)$ and
\[
\xymatrix{
0 \ar[r] & \Omega M \ar[r]^g & P_0 \ar[r]^f & M \ar[r] & 0
}
\]
be an exact sequence with finitely generated projective $P_0$. Then
the induced sequence
\[
\xymatrix{
0 \ar[r] & (\Omega M)^\wedge \ar[r]^-{g^\wedge} & P_0^\wedge \ar[r] &
K_{LR}(M) \ar[r] & 0}
\]
is a projective resolution of $K_{LR}(M)$. If $P_0$ is also injective,
then $K_{LR}(M)$ is rigid. 
\end{itemize}
\end{lemma}

We refer to section~\ref{ss:non-rigid-inter-ext} for an example
where $K_{LR}(M)$ is not rigid.

\begin{proof} a) For $L$ in $\mod(\cp)$, we have functorial isomorphisms
\[
(K_R M)(L) = \Hom(L^\wedge, K_R M) = \Hom(\res(L^\wedge), M) =
\Hom(L,M) = M^\wedge(L).
\]

b) We have $K_L(P)=P^\wedge$ and by a), we have $P^\wedge=K_R(P)$.

c) We have a commutative diagram with exact rows
\[
\xymatrix{
            &                                          & 
                                     K_L(P_0) \ar@{=}[d] \ar[r] & K_L(M) \ar[r] \ar[d] & 0 \\
0 \ar[r] & (\Omega M)^\wedge \ar[r] & P_0^\wedge \ar[r]^-{f^\wedge} & K_R(M)=M^\wedge & 
}
\]
Thus, the image of $f^\wedge$ is $K_{LR}(M)$. Now suppose that $P_0$
is injective. Let $h: (\Omega M)^\wedge \to K_{LR}(M)$ represent
an element of $\Ext^1(K_{LR}(M),K_{LR}(M))$. 
\[
\xymatrix{
0 \ar[r] & (\Omega M)^\wedge \ar@{-->}[dl]_l \ar[r]^-{g^\wedge} \ar[d]^h  & 
              P_0^\wedge \ar[r]     & K_{LR}(M) \ar[r] & 0 \\
P_0^\wedge \ar[r] & K_{LR}(M) }
\]
Since $(\Omega M)^\wedge$
is projective, the morphism $h$ lifts along $P_0^\wedge \to K_{LR}(M)$
to a morphism $l: (\Omega M)^\wedge \to P_0^\wedge$. Since 
$P_0^\wedge$ is injective, the morphism $l$ extends along
$g^\wedge: (\Omega M)^\wedge \to P_0^\wedge$. Thus, the
morphism $h$ extends along $g^\wedge$ and its class
in $\Ext^1(K_{LR}(M),K_{LR}(M))$ vanishes.
\end{proof}

\section{Desingularization of quiver Grassmannians}
\label{s:desingularization}

\subsection{Repetition quivers and the derived category}
\label{ss:repetition-quivers-and-Happels-theorem}
Let $Q$ be a quiver. We write $Q_0$ for its set of vertices and
$Q_1$ for its set of arrows. We assume that $Q$ is {\em finite},
i.e. both $Q_0$ and $Q_1$ are finite, and
{\em acyclic}, i.e. it has no oriented cycles.

Let $k$ be a field. The path algebra $kQ$ is 
a finite-dimensional hereditary $k$-algebra. 
Let $\mod kQ$ be the category of all 
$k$-finite-dimensional right $kQ$-modules. 
The projective indecomposable modules are given up to isomorphism 
by $P_i=e_ikQ$, where $e_i$ denotes the path of length zero at the vertex $i$. 
The head of $P_i$ is the simple module $S_i$ concentrated at the vertex~$i$.

We denote by $\cd_Q$ the bounded derived category $\cd^b(\mod kQ)$. 
Endowed with the  shift (=suspension) functor $\Sigma$
it is a triangulated category. 
By \cite{Happel87} the derived category $\cd_Q$ is a Krull--Schmidt category which 
admits Auslander-Reiten triangles or, equivalently, a Serre functor, 
cf. \cite{ReitenVanDenBergh02}. Let $\tau_{\cd_Q}$
be the Auslander-Reiten translation. 
The Serre functor is then given by $S=\Sigma \circ \tau_{\cd_Q}$ and
is isomorphic to the derived tensor product with 
the bimodule $D(kQ)=\Hom_k(kQ,k)$. 
Let us denote by $\ind(\cd_Q)$ a full subcategory of $\cd_Q$ whose 
set of objects contains exactly one representative of each isomorphism class of indecomposable objects of $\cd_Q$.
If $Q$ is an orientation of an ADE Dynkin diagram, Happel showed that $\ind(\cd_Q)$ can be fully described in
combinatorial terms using the so-called repetition quiver. 
The {\em repetition quiver $\Z Q$}, cf. \cite{Riedtmann80a}, has the set of 
vertices $Q_0 \times \Z$. We obtain its set of arrows from $Q_1$ as follows: 
For each arrow $\alpha: i \to j$ in $Q_1$ and each integer $p$, we have an arrow 
$(\alpha, p): (i,p) \to (j,p)$ and an arrow $\sigma(\alpha,p): (j,p-1) \to (i,p)$. 
We define the automorphism $\tau$ of $\Z Q$ to be the shift
by one unit to the left, so that we have in particular
$\tau(i,p)=(i, p-1)$ for all vertices $(i, p) \in Q_0 \times \Z$.

Following \cite{Gabriel80} \cite{Riedtmann80}, we
define the {\em mesh category $k(\Z Q)$} to be the $k$-category whose
objects are the vertices of $\Z Q$ and whose morphism space from
$a$ to $b$ is the space of all $k$-linear combinations of paths from
$a$ to $b$ modulo the subspace spanned by all elements $u r_x v$, where
$u$ and $v$ are paths and $r_x$ is the sum of all paths from $\tau(x)$ to $x$.
For example if $Q=\vec{A}_2: 1 \to 2$, the repetition quiver is
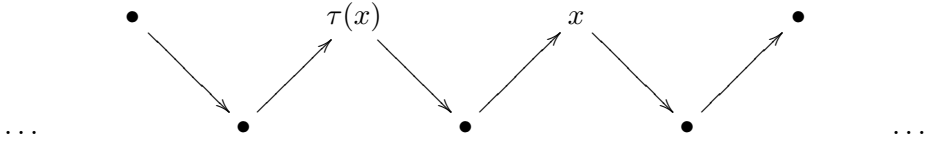
\[
\xymatrix@=0.5cm@!{
 & \bt \ar[dr] & & \tau(x) \ar[dr] & & x \ar[dr] & & \bt &  \\
 \ldots & & \bt \ar[ur]   & & \bt \ar[ur]  & & \bt \ar[ur]  & & \ldots 
  }
\]
In the mesh category $k(\Z \vec{A}_2)$ associated with
the quiver $Q=\vec{A}_2: 1 \to 2$, the composition of any two
consecutive arrows vanishes. 
The mesh category $k(\Z Q)$ and $\ind(\cd_Q)$ are related as follows:

\begin{theorem}[Prop.~4.6 of \cite{Happel87}]
\label{thm:Happel}
There is a canonical fully faithful functor
\[
H: k(\Z Q) \to \ind(\cd_Q)
\]
taking each vertex $(i,0)$ to the indecomposable
projective module $P_i$, $i\in Q_0$. It is an equivalence iff
$Q$ is an orientation of an ADE Dynkin diagram. 
\end{theorem}
Note that the map $\tau$ induces naturally an autoequivalence on $k(\Z Q)$. 
Happel showed that $H \circ \tau$ is isomorphic to $\tau_{\cd_Q}\circ H$.  
We will therefore denote $\tau_{\cd_Q} $ by $\tau$. 

\subsection{The regular and the singular Nakajima category}
\label{ss:the-Nakajima-categories}
Let $Q$ be a finite acyclic quiver as in section~\ref{ss:repetition-quivers-and-Happels-theorem}.
The {\em framed quiver $\tilde{Q}$} is obtained from $Q$ by adding, for each vertex $i$, a new vertex $i'$ and
a new arrow $i \to i'$. For example, if $Q$ is the quiver $1 \to 2$, the
framed quiver is
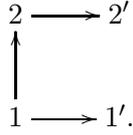
\[ 
\xymatrix{ 
2 \ar[r] & 2' \\
1 \ar[r] \ar[u] & 1' .
} 
\]
Let $\Z\tilde{Q}$ be the repetition quiver of $\tilde{Q}$. We refer to the
vertices $(i', p)$, $i\in Q_0$, $p\in \Z$, as the {\em frozen vertices} of $\Z \tilde{Q}$
and mark them by squares .
For example, if the underlying quiver of $Q$ is the Dynkin diagram $A_2$, the repetition $\Z\tilde{Q}$ is the quiver
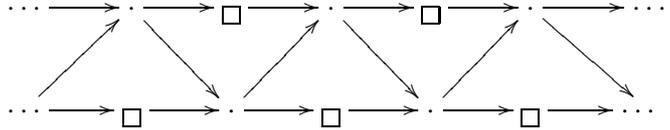
\[ 
\xymatrix{ \cdots  \ar[r]  & \cdot \ar[r] \ar[rd] & \boxed{ } \ar[r]  &  \cdot \ar[r]   \ar[rd] &  \boxed{ } \ar[r] & \cdot \ar[r]   \ar[rd] & \cdots \\
\cdots  \ar[r]  \ar[ru]  & \boxed{}  \ar[r]  &  \cdot  \ar[r]  \ar[ru]  &  \boxed{ } \ar[r]  & \cdot  \ar[r]  \ar[ru]  & \boxed{} \ar[r]  &  \cdots\ . }
\]
For a vertex $x=(i,p)$, we put $\sigma(x)=(i', p-1)$
and for a vertex $(i', p)$, we put $\sigma(i',p)=(i, p)$.

The {\em regular Nakajima category $\cR$} is the mesh
category $k(\Z\tilde{Q})$, where we only impose mesh relations associated with the {\em non frozen vertices}. The {\em singular Nakajima category
$\cs$} is the full subcategory of $\cR$ whose objects are the frozen 
vertices. 

Notice that while $\cR$ is given by a quiver with relations, it
is not clear how to describe the subcategory $\cs\subset\cR$
in this way. In Theorem~2.4 of \cite{KellerScherotzke13a}, we have shown that 
$\cs$ is given by a quiver $Q_\cs$ with relations such that the vertices of $Q_\cs$ 
are the frozen vertices of $\Z \tilde{Q}$, 
the number of arrows in $Q_{\cs}$ from $\sigma(x)$ to $\sigma(y)$ 
equals the dimension of 
\[
\Ext^1_{\cd_Q}(H(y), H(x)) 
\]
and the minimal number of relations in the space of paths from $\sigma(x)$ to
$\sigma(y)$ is given by the dimension of
\[
\Ext^2_{\cd_Q} (H(y), H(x)).
\]
For the quiver $Q: 1\to 2$, we find that $Q_\cs$ is the quiver
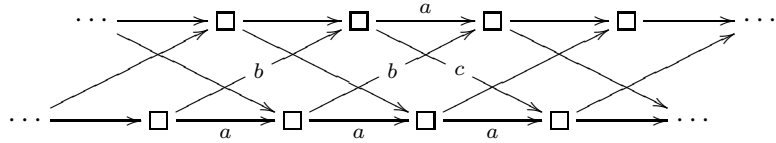
\[
\begin{xy} 0;<0.5pt,0pt>:<0pt,-0.5pt>:: 
(-50,0) *+{\cdots} ="-1",
(-100,75) *+{\cdots}="-2",
(0,75) *+{\boxed{}} ="0",
(50,0) *+{\boxed{}} ="1",
(100,75) *+{\boxed{}} ="2",
(150,0) *+{\boxed{}} ="3",
(200,75) *+{\boxed{}} ="4",
(250,0) *+{\boxed{}} ="5",
(300,75) *+{\boxed{}} ="6",
(350,0) *+{\boxed{}} ="7",
(400,75) *+{\cdots} ="8",
(450,0) *+{\cdots} ="9",
"-1", {\ar "1"},
"-1", {\ar "2"},
"-2", {\ar "0"},
"-2", {\ar "1"},
"0", {\ar_a "2"},
"0", {\ar|-*+{{\scriptstyle b}} "3"},
"1", {\ar"3"},
"1", {\ar"4"},
"2", {\ar_a "4"},
"2", {\ar|-*+{{\scriptstyle b}} "5"},
"3", {\ar^a "5"},
"3", {\ar|-*+{{\scriptstyle c}} "6"},
"4", {\ar_a "6"},
"4", {\ar"7"},
"5", {\ar"7"},
"5", {\ar"8"},
"6", {\ar"8"},
"6", {\ar"9"},
"7", {\ar"9"},
\end{xy}
\]
and that $\cs$ is isomorphic to the path category of $Q_\cs$ modulo the ideal
generated by all relations of the form $ab-ba$, $ac-ca$, $a^3-cb$ and $bc-a^3$ 
(we denote all horizontal arrows by $a$, all rising arrows by $b$ and all 
descending arrows by $c$).

\subsection{Graded quiver varieties} 
\label{ss:graded-quiver-varieties} Although, from a strictly
logical point of view, we do not
need graded quiver varieties in this article, we include this
section for the convenience of the reader.
Let us fix a dimension vector $w: \cs_0 \to \N$, i.e.
a function with finite support.
The {\em affine graded quiver variety $\cm_0(w)$} is
the affine variety $\rep_w(\cs)$ of $\cs$-modules $M$ 
such that $M(u)=k^{w(u)}$ for each vertex $u\in \cs_0$. 
This definition is equivalent to Nakajima's original 
definition in \cite{Nakajima01} \cite{Nakajima11} ($Q$ bipartite) and 
to the definition in \cite{Qin12} ($Q$ acyclic), cf. the proof of Theorem~2.4 of \cite{LeclercPlamondon12},  
based on \cite{Lusztig98a} \cite{LebruynProcesi90}.

Now in addition to the dimension vector $w: \cs_0 \to \N$, let us
fix a dimension vector $v: \cR_0\setminus \cs_0\to \N$ of $\cR$.
Let $\rep_{v,w}(\cR)$ denote the variety of $\cR$-modules of dimension 
vector $(v,w)$.  Let $G_v$ be the product of the groups $\GL(k^{v(x)})$, 
where $x$ runs through the non frozen vertices. The group $G_v$ acts 
on the variety $\rep_{v,w}(\cR)$ by base change in the spaces
$k^{v(x)}$. To define the smooth graded quiver variety $\cm(v,w)$, we consider 
the set $\tilde{\cm}(v,w) \subset \rep_{v,w}(\cR)$ formed by the $\cR$-modules 
$M$ with dimension vector $(v,w)$ which are {\em stable}, i.e. do not have 
non zero $\cR$-submodules which restrict to the zero module of $\cs$.  
The {\em graded quiver variety} is the quotient 
\[
\cm(v,w) =\tilde{\cm}(v,w)/G_v. 
\]
For this definition and the following facts, we refer to Nakajima's 
work \cite{Nakajima01} \cite{Nakajima11} for the case where $Q$ is Dynkin 
or bipartite and to Qin \cite{Qin12} and Kimura-Qin \cite{KimuraQin12} for 
the extension to the case of an arbitrary acyclic quiver $Q$.
The set  $\cm(v,w)$ canonically becomes a smooth quasi-projective variety and the projection map
\[
\pi : \cm(v,w) \to \cm_0(w)
\]
taking an $\cR$-module $M$ to its restriction $\res M$ is a proper map 
(notice that $\res$ is constant on the $G_v$-orbits). 
We denote by $\cm(w)=\coprod_v \cm(v,w)$  the disjoint union over all dimension 
vectors $v$. 

In \cite{Nakajima01} Nakajima shows that the affine quiver variety admits a finite stratification
\[
\cm_0(w)= \coprod_v  \cm_0^{bs}(v,w) 
\]
into the locally closed smooth subsets $\cm^{bs}_0(v,w)$ formed by
the orbits of {\em bistable} (i.e. stable and co-stable) representations
(these are called `regular' in \cite{Nakajima01}).
He also shows that we have
\[
\overline{\cm^{bs}_0(v,w)} = \coprod_{v' \le v} \cm_0^{bs}(v',w) ,
\]
where the order on the dimension vectors is given by 
$v' \le v$ if and only if $v'(i) \le v(i)$ for all $ i \in \cR_0\setminus \cs_0$.
As shown in \cite{KellerScherotzke13a}, the strata $\cm_0^{bs}(v,w)$ 
can be described by 
\[
\cm_0^{bs}(v,w) =\{ L \in \cm_0(w) | \dimv K_{LR} L =(v,w) \} ,
\]
where $K_{LR}:\Mod \cs \to \Mod \cR$ is the intermediate Kan extension 
in the sense of Section~\ref{s:intermediate-extensions} associated
with the restriction functor $\res: \Mod \cR \to \Mod \cs$. 

\subsection{Configurations} \label{ss:configurations}
Let $C$ be a subset of the set of vertices of the repetition quiver $\Z Q$. 
Let $\cR_C$ be the quotient of $\cR$ by the ideal generated by the identities 
of the frozen vertices not belonging to $\sigma^{-1}(C)$ and let $\cs_C$ be the
full subcategory of $\cR_C$ formed by the vertices in $\sigma^{-1}(C)$. 
Notice that $\Mod(\cR_C)$ is a subcategory
of $\Mod(\cR)$ and similarly $\Mod(\cs_C)$ a subcategory of
$\Mod(\cs)$. Let $\res^C: \Mod(\cR_C) \to \Mod(\cs_C)$
be the restriction functor. Clearly, it is just the restriction of the 
functor $\res: \Mod(\cR) \to \Mod(\cs)$ to the subcategories under
consideration. The left and right adjoints $K_L$ and $K_R$ of
$\res$ take the subcategory $\Mod(\cs_C)$ of $\Mod(\cs)$ to
$\Mod(\cR_C)$ and thus induce left and right adjoints
$K_L^C$ and $K_R^C$ of $\res^C$ so that we have 
\[
\xymatrix{
\Mod(\cR_C) \ar[d]|-*+{{\scriptstyle \res^C}} \\
\Mod(\cs_C). \ar@<3ex>[u]^{K^C_L} \ar@<-3ex>[u]_{K_R^C}
}
\]
 The functor $\res^C$ is a localization of abelian categories in the sense of \cite{Gabriel62}, or equivalently, its adjoints are fully faithful.
 {\em In the sequel, we will omit the exponents $C$ in the 
notation for the functors $K^C_L$ and $K_R^C$ and simply write $K_L$ and
$K_R$}.

To ensure that the category of finite-dimensional $\cR_C$-modules has 
global dimension at most two, we make the following assumption on $C$.
\begin{assumption} \label{main-assumption}
For each non frozen vertex $x$ of $\Z \tilde{Q}$, the sequences
\begin{equation} \label{eq:left-exact-sequences}
0 \to \cR_C(?,x) \to \bigoplus_{x \to y} \cR_C(?,y) \quad \mbox{and}\quad
0 \to \cR_C(x,?) \to \bigoplus_{y \to x} \cR_C(y,?) 
\end{equation}
are exact, where the sums range over all arrows of $\Z \tilde{Q}$ whose
source (respectively, target) is $x$.
\end{assumption}
Note that the assumption holds if $C$ is the set of all vertices of $\Z Q$.
The following situation provides further examples of sets $C$ satisfying 
the assumption: Assume that $\ce$ is a 
$\Hom$-finite exact Krull--Schmidt category which is standard 
(in the sense of section~2.3, page~63 of \cite{Ringel84}) and
whose stable Auslander--Reiten quiver is $\Z Q$. 
Let us define $C$ as the set of vertices $c$ such that $\sigma^{-1}(c)$ corresponds
to a projective indecomposable object of $\ce$.
Then the sequences (\ref{eq:left-exact-sequences}) are associated with
Auslander--Reiten conflations of $\ce$ and hence are exact.
In section~\ref{ss:piecewise-hereditary}, we show how iterated tilted
algebras of Dynkin type give rise to such configurations $C$.


In fact, we have shown in Theorem~5.23 of \cite{KellerScherotzke13a} 
that when the assumption holds and $Q$ is a Dynkin quiver, then the
set $C$ always comes from the choice of a $\Hom$-finite exact
Krull--Schmidt category which is standard and whose stable 
Auslander--Reiten quiver is $\Z Q$.

\subsection{The desingularization theorem} 
\label{ss:proof-of-desingularization}

Let $M$ be a finite-dimensional $\cs_C$-module such that $K_{LR}(M)$ is rigid.
Recall that a variety is {\em equidimensional} if all its irreducible components
have the same dimension. 

\begin{lemma} \label{lemma:smooth-Grassmannian}
Each quiver Grassmannian $\Gr_e(K_{LR}(M))$  is smooth and
equi\-dimensional. 
\end{lemma}

\begin{proof} 
Indeed, by Proposition~7.1 of \cite{CerulliFeiginReineke12a}, 
we only need to check the following:
The module $K_{LR}(M)$ is finite-dimensional, 
the space $\Ext^i(K_{LR} M, K_{LR} M)$ vanishes for all 
$i \ge 1$ and the category of finite-dimensional $\cR_C$-modules
 is of global dimension at most $2$. 
The first condition is satisfied by section~4.8 of \cite{KellerScherotzke13a} 
and the last one is shown in Lemma~3.5 of \cite{KellerScherotzke13a}.
Finally, the module $K_{LR}(M)$ is both of projective and of injective dimension 
at most one by Lemma~4.15 of \cite{KellerScherotzke13a} and 
$K_{LR}(M)$ is rigid by our assumption.
Therefore $\Ext^i(K_{LR} (M), K_{LR} (M))$ vanishes for all $i \ge 1$.  
\end{proof}

We introduce a stratification with finitely 
many strata on $\Gr_w(M)$ using Nakajima's stratification of the
representation spaces $\cm_0(w)$.
Following \cite[2.3]{CerulliFeiginReineke12}, we write 
\[
\Hom^0(w, M)=\{ (N, f) | N \in \mod_w (\cs_C) \text{ and } f: N \to M \text{ is injective} \}.
\]
Then $\Gr_w(M)$ is isomorphic to the quotient $\Hom^0(w, M) / G_w$ where $G_w$ is the product of the groups 
GL$(k^{w(x)})$ for all $x \in \cs_0$. 
We have a canonical map
\[
p: \Hom^0(w, M) \to \mod_w (\cs_C) 
\]
given by the projection. For a function with finite support 
$v: \cR_0 \setminus \cs_0 \to \N$, we define the locally closed subset 
$\cm_0^{bs}(v,w)^{\Gr}$ of $\Gr_w(M)$ by
\[
\cm_0^{bs}(v,w)^{\Gr}= p^{-1} ( \cm_0^{bs}(v,w))/G_w.
\]
For fixed $w$, the subset $\cm_0^{bs}(v,w)$ is non empty only for finitely
many functions $v$. Thus, the variety $\Gr_w(M)$ decomposes into the disjoint 
union of finitely many strata $ \cm_0^{bs}(v,w)^{\Gr}$. 
We have 
\begin{equation} \label{eq:closure-of-strata}
\overline{\cm_0^{bs}(v,w)^{\Gr}} \subset  
p^{-1}( \overline{\cm_0^{bs}(v,w) })/G_w = 
\coprod_{v' \le v}\cm_0^{bs}(v',w)^{\Gr} .
\end{equation}
Thus, if $v$ is minimal with the property that $\cm_0^{bs}(v,w)^{\Gr}$ is not empty, 
then it is closed. 

\begin{lemma} \label{lemma:definition-vc} Let $C$ be an irreducible component
of $\Gr_w(M)$. Then there is a unique dimension vector $v_C$ such that
\[
C\cap\cm_0^{bs}(v_C,w)^{\Gr}
\]
is an open dense subset in $C$. The vector $v_C$ is the unique maximal
element in the set of vectors $v$ such that 
\[
C\cap \cm_0^{bs}(v,w)^{\Gr}
\]
is non empty.
\end{lemma}

\begin{proof} 
Let $V$ be the set of functions $v$ such
that $C\cap \cm_0^{bs}(v,w)^{\Gr}$ is not empty. Then it follows from
(\ref{eq:closure-of-strata}) that for a
subset $V'\subset V$ stable under taking predecessors for the componentwise order, 
the union of the subsets $C\cap \cm_0^{bs}(v,w)^{\Gr}$, where $v$ ranges over $V'$,
is closed in $C$.  In particular, this happens if $V'$ is the complement
of a maximal element $v$ of $V$. Thus, if $v_C$ is maximal in $V$, the
set $C\cap \cm_0^{bs}(v_C,w)^{\Gr}$ is open and dense in $C$. Since
the strata $\cm_0^{bs}(v,w)^{\Gr}$ are pairwise disjoint, the maximal
element $v_C$ is uniquely determined by the irreducible component $C$.
\end{proof}

As $K_{LR} (M)$ is stable, each submodule of $K_{LR}(M)$ is stable and
in particular, the points of $\Gr_{(v,w)}(K_{LR}(M))$ yield points of $\cm(v,w)$. 
We define the subset $\cm^{bs}(v,w)^{\Gr}$ of $\Gr_{(v,w)}(K_{LR}(M))$
in analogy with $\cm_0^{bs}(v,w)^{\Gr}$, so that a point of $\Gr_{(v,w)}(K_{LR}(M))$
lies in $\cm^{bs}(v,w)^{\Gr}$ if and only if the corresponding submodule
is bistable. 

\begin{lemma} \label{lemma:isom-on-bistables}
\begin{itemize}
\item[a)]
The restriction functor induces an isomorphism
\[
\cm^{bs}(v,w)^{\Gr} \iso \cm^{bs}_0(v,w)^{\Gr}.
\]
\item[b)] The varieties $\cm^{bs}(v,w)^{\Gr}$, 
$\overline{\cm^{bs}(v,w)^{\Gr}}$ and  $\cm^{bs}_0(v,w)^{\Gr}$
are smooth and equidimensional.
\end{itemize}
\end{lemma}

\begin{proof} a) 
Let us first check that this map is bijective: 
Indeed, it is surjective, since a submodule $L \subset M$ is 
obtained by restricting the bistable submodule $K_{LR}(L)\subset K_{LR}(M)$.
It is injective because if a submodule $N \subset K_{LR}(M)$ is costable, it
is generated by the spaces $N(x)$, $x\in \cs_0$, and is
thus determined by its restriction to $\cs$. This also shows how
to construct an inverse of the map: a submodule $L\subset M$
is sent to the submodule of $K_{LR}(M)$ generated by the
spaces $L(x)$, $x\in \cs_0$.

b) By Lemma~\ref{lemma:smooth-Grassmannian}, the variety
$\Gr_{(v,w)}(K_{LR}(M))$ is smooth and equidimensional. Thus,
the same holds for its open subset $\cm^{bs}(v,w)^{\Gr}$.
The closure of this subset is the union of the connected (=irreducible)
components of $\Gr_{(v,w)}(K_{LR}(M))$ which meet $\cm^{bs}(v,w)^{\Gr}$.
Thus, the closure is also smooth and equidimensional.
By~a), we obtain the same assertion for $\cm_0^{bs}(v,w)^{\Gr}$.
\end{proof}

As in section~\ref{s:intro}, we say that a morphism of
algebraic varieties $\pi: X \to Y$ is a {\em desingularization}
if $X$ is smooth, $\pi$ is proper and surjective and induces an
isomorphism from an open dense subset of $X$ onto
an open dense subset of $Y$. Recall that the {\em bistable
quiver Grassmannian} is defined as
\[
\Gr^{bs}_{(v,w)}(K_{LR}(M))=\overline{\cm^{bs}(v, w)^{\Gr}}.
\]

\begin{theorem} \label{thm:desingularization}
As above, we assume that $M$ is an $\cs_C$-module
such that $K_{LR}(M)$ is rigid. Let $w$ be a dimension vector less or equal
to the dimension vector of $M$. Let $\cv_w(M)$ be the set of the
vectors $v_C$, where $C$ ranges over the irreducible components
of $\Gr_w(M)$ (cf. Lemma~\ref{lemma:definition-vc}). Let 
\[
\pi^{bs} : \coprod_{v\in \cv_w(M)} \Gr^{bs}_{(v,w)}(K_{LR}(M)) \to \Gr_w(M),
\]
be the map taking a submodule $L$ to its restriction to $\cs_C$.
\begin{itemize}
\item[a)] The map $\pi^{bs}$
is a desingularization. It induces an isomorphism between the dense open
subsets
\[
\coprod_{v\in\cv_w(M)} \cm^{bs}(v,w)^{\Gr} \to \coprod_{v\in\cv_w(M)}
\cm^{bs}_0(v,w)^{\Gr}  \subset \Gr_w(M).
\]

\item[b)] Let $C$ be an irreducible component of $\Gr_w(M)$ and $v_C$ the
unique vector such that $C\cap \cm_0^{bs}(v_C,w)^{\Gr}$ is a dense open
subset of $C$ (Lemma~\ref{lemma:definition-vc}).
Then the map
\[
\pi_C : \pi^{-1}(C)\cap  \Gr^{bs}_{(v,w)}(K_{LR}(M)) \to C
\]
taking a submodule $L$ to its restriction $\res(L)$ to $\cs_C$ is a desingularization. 
It induces an isomorphism between the dense open subsets
\[
\pi^{-1}(C)\cap \cm^{bs}(v_C,w)^{\Gr} \iso C\cap \cm_0^{bs}(v_C,w)^{\Gr}.
\]
\end{itemize}
\end{theorem} 

\begin{remark} Theorems~\ref{thm:intro-pre-desingularization} and 
\ref{thm:intro-desingularization} of Section~\ref{s:intro} are immediate
consequences of part a).
\end{remark} 

\begin{proof} Part a) is an immediate consequence of part b). 
To prove b), we note that the domain of $\pi_C$ is smooth by part b) of 
Lemma~\ref{lemma:isom-on-bistables}. The map $\pi_C$ is proper since 
its domain is projective. It induces the isomorphism between dense open sets 
by part a) of Lemma~\ref{lemma:isom-on-bistables}.
\end{proof}

Generalizing remark~7.8 of \cite{CerulliFeiginReineke12a} we conjecture that 
$\overline{\cm^{bs}(v, w)^{\Gr}}$ equals the whole Grassmannian 
$\Gr_{(v,w)}(K_{LR} M)$. 
If this is true, we have an easy description of the fibres
using the next theorem.

\begin{theorem} \label{thm:fibre-of-piGr}
The fibre of $\pi_{v,w}: \Gr_{(v,w)}(K_{LR} M) \to \Gr_w(M)$ over a submodule $U\subset M$ is isomorphic
to the quiver Grassmannian of 
submodules
of dimension $(v,w) - \dimv K_{LR}(U)$ of the module
$K_{LR}(M)/K_{LR}(U)$.
\end{theorem}

\begin{proof} The claim is equivalent to the statement that
the fibre is isomorphic to the variety of submodules $V\subset K_{LR}(M)$ containing
$K_{LR}(U)$  and such that $\dimv V=(v,w)$.

By the definition of $\pi_{v,w}$, this is equivalent
to the following statement: Suppose
that $V\subset K_{LR}(M)$ is a submodule of dimension vector
$(v,w)$. Then the restriction $\res(V)$ equals $U$ if
and only if we have $K_{LR}(U) \subset V$.

Indeed, if we have $K_{LR}(U)\subset V$, then we have
\[
U = \res K_{LR}(U) \subset \res(V)
\]
and 
\[
\dimv U =  \dimv \res(V).
\]
Hence we have $U=\res(V)$ as claimed.

Conversely, suppose that we have $\res(V)=U$. By assumption, we
have $V\subset K_{LR}(M)$. Since $K_{LR}(M)$ is stable, so is $V$. Thus,
we have $K_{LR}(\res(V)) \subset V$. But we also have $K_{LR}(V)= K_{LR}(U)$.
Thus, we have $K_{LR}(U)\subset V$ as claimed.
\end{proof}

\subsection{Example of a non rigid intermediate extension}
\label{ss:non-rigid-inter-ext}
In section~\ref{ss:proof-of-desingularization}, to prove that $\pi_{\Gr}$ is 
indeed a desingularization, we made the assumption 
that the intermediate extension $K_{LR} (M)$ is rigid.
We gave some sufficient conditions for this to hold in  
Lemma~\ref{lemma:Auslander-category}.
Let us show by an example that $K_{LR}(M)$ is not always rigid.

Let $Q$ be some orientation of $A_3$ and let $w$ be a dimension vector which takes 
the value $1$ in the marked boxes and zero everywhere else in the quiver of 
$\cR$ associated with $Q$:
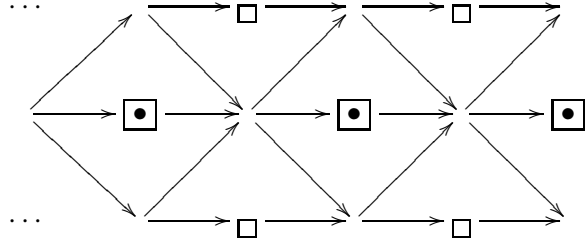
\[ \xymatrix{  \cdots
&   \ar[rd] \ar[r] & \boxed{} \ar[r] &  \ar[rd] \ar[r] & \boxed{} \ar[r]& \\
\ar[r] \ar[ru] \ar[rd] & \boxed{\bt } \ar[r]  &  \ar[r] \ar[dr]  \ar[ru] & \boxed{\bt } \ar[r] & \ar[r] \ar[dr] \ar[ru] & \boxed{\bt }  \\
\cdots & \ar[ru] \ar[r] & \boxed{ } \ar[r] & \ar[ru]  \ar[r] & \boxed{ } \ar[r]&
}
\]
By Theorem~2.4 of \cite{KellerScherotzke13a}, the modules in $\cs$ 
with dimension vector 
$w$ are given by the representations with dimension 
vector $(1,1,1)$ of the quiver  
$$ 
\xymatrix{ 
\bt \ar@/^2ex/[rr]  \ar[r] & \bt  \ar[r] & \bt
}
$$
The following module $M \in \mod \cs$  is not rigid and has dimension vector $w$:
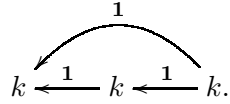
$$ 
\xymatrix{ 
k  & k \ar[l]_{\id}  & k.\ar[l]_{\id} \ar@/_2pc/[ll]_{\id}
}
$$
Its intermediate extension $K_{LR} M$ is  the $\cR_C$-module given by
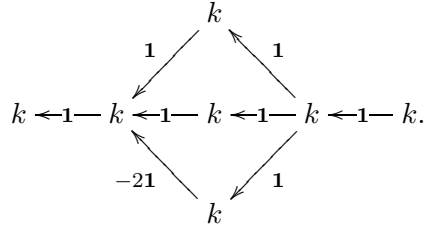
\[ \xymatrix{  
&  & k\ar[ld]_{\id} && \\
k  &k \ar[l]|{\id} & k \ar[l]|{\id} & k\ar[l]|{\id} \ar[lu]_{\id} \ar[ld]^{\id}  & k. \ar[l]|{\id}   \\
& & k\ar[lu]^{-2 \id}   & & 
}\]
Using Corollary~3.6 of \cite{KellerScherotzke13a} it is easy to see that the space
$\Ext^1(S_x,K_{LR} (M))$ vanishes for all non-frozen vertices $x\in \cR_0-\cs_0$ 
except for the vertex $z$:
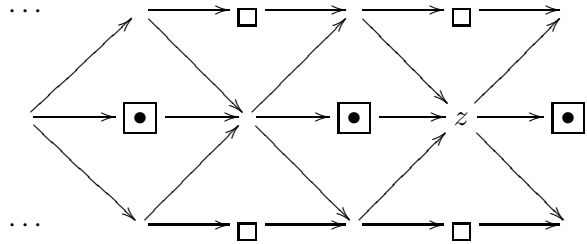
\[ 
\xymatrix{  \cdots
&   \ar[rd] \ar[r] & \boxed{} \ar[r] &  \ar[rd] \ar[r] & \boxed{} \ar[r]& \\
\ar[r] \ar[ru] \ar[rd] & \boxed{\bt } \ar[r]  &  \ar[r] \ar[dr]  \ar[ru] & \boxed{\bt } \ar[r] & z \ar[r] \ar[dr] \ar[ru] & \boxed{\bt }  \\
\cdots & \ar[ru] \ar[r] & \boxed{ } \ar[r] & \ar[ru]  \ar[r] & \boxed{ } \ar[r]&
}
\]

Hence, by Lemma~4.13 of \cite{KellerScherotzke13a}, the cokernel of $K_{LR} M \hookrightarrow K_R M$ is represented as a module
of $\cd_Q \cong \cR/\langle \cs \rangle $ by $z$, i.e. it is the 
injective module $z_{\cd}^{\vee}$ defined by
$$z_{\cd}^{\vee}= D\Hom_{\cd_Q}(z, -),$$ 
where we identify the vertex $z$ with its image in $\cd_Q$ under the Happel functor.
Now using the projective resolution 
of $z_{\cd}^{\vee}$ from Theorem~3.7 of \cite{KellerScherotzke13a}, we have 
$$\Ext^2_{\cR}(z_{\cd}^{\vee},z_{\cd}^{\vee})\cong \Hom((\Sigma^{-1} z)^{\vee} , z_{\cd}^{\vee}) \cong D\Hom_{\cd_Q}(z, \Sigma^{-1} z)$$
which vanishes, since $z$ is an indecomposable object of $\cd_Q$.
So using part c) of Lemma~\ref{lemma:necessary-condition-rigidity}, we deduce that 
$K_{LR} (M)$ is not rigid, since $M$ is not rigid. 

\section{Quiver Grassmannians over repetitive algebras}
\label{s:quiver-grassmannians-over-repetitive-algebras}

\subsection{Repetitive algebras} \label{ss:repetitive-algebras}
Let $A$ be a finite-dimensional $k$-algebra and $DA=\Hom_k(A,k)$ the
bimodule dual to $A$. Let $T(A)$ be the {\em trivial extension}, i.e. the
algebra $A\oplus DA$ with the multiplication
\[
(a,f)(b,g)=(ab, ag+fb), a,b\in A, \;\; f,g \in DA.
\]
We endow $T(A)$ with the $\N$-grading such that $T(A)^0=A$, $T(A)^1=DA$
and $T(A)^p=0$ for all $p\geq 2$. A $\Z$-graded module $M$ over $T(A)$ is given
by a sequence $M^p$, $p\in\Z$, of $A$-modules and $A$-linear maps
\[
m^p: \nu(M^p) \to M^{p+1} \ko p\in\Z\ko
\]
where $\nu(L)=L\ten_A DA$, such that $m^{p+1}\circ \nu(m^p)=0$ for all
$p\in\Z$. Equivalently, such a module may be interpreted as a module
over the {\em repetitive algebra}, which is a suitably defined (locally unital)
infinite matrix algebra, cf. section~10 of \cite{Happel87}. Let $\grm(T(A))$
be the category of $\Z$-graded $T(A)$-modules of finite dimension over $A$
with morphisms the homogeneous $T(A)$-linear maps of degree $0$.
Let $\cp\subset \grm(T(A))$ be the full subcategory of the projective
graded modules. We have a canonical equivalence
\[
\grm(T(A)) \iso \mod(\cp)
\]
taking a module $M$ to the restriction of $\Hom(?,M)$ to $\cp$ 
and the category $\cp$ fits into the setup of section~\ref{ss:Auslander-category}:
it is coherent since its projectives are of finite dimension and the finitely
presented $\cp$-modules coincide with the finite-dimensional $\cp$-modules.

For a graded $T(A)$-module $M$, let $M\langle 1\rangle$ denote the
shifted module defined by
\[
M\langle 1\rangle^p = M^{p+1}, m^p_{M\langle 1\rangle}= m_M^{p+1}.
\]
For a graded left $T(A)$-module $M$, let $DM$ denote the $k$-dual right
module with $(DM)^p= D(M^{-p})$, $p\in \Z$. We have a canonical isomorphism
of $A$-modules
\[
DT(A) \iso (T(A))\langle 1 \rangle.
\]
This shows that the projectives coincide with the injectives in
$\mod(\cp)$, which is therefore a Frobenius category.
Thus, the associated stable category $\ul{\mod}(\cp)$ (the quotient
of $\grm T(A)$ by the ideal of morphisms factoring through a projective)
is canonically triangulated, cf.~\cite{Happel87}. By a theorem of
Happel \cite{Happel87}, the canonical embedding
\[
\mod A \to \ul{\mod}(\cp)
\]
taking an $A$-module to the corresponding graded $T(A)$-module concentrated
in degree $0$ extends to a fully faithful triangle functor
\begin{equation} \label{eq:derived-eq-stable}
\cd^b(\mod A) \to \ul{\mod}(\cp)
\end{equation}
and this functor is an equivalence if and only if $A$ is of finite global dimension.

\subsection{The case of iterated tilted algebras of Dynkin type}
\label{ss:piecewise-hereditary}
From now on, let us suppose that $A$ is derived equivalent to the path
algebra $kQ$ for a Dynkin quiver $Q$ with underlying graph $\Delta$
(for example, we can take $A=kQ$ or $A$ a tilted algebra of type $Q$,
cf.~\cite{Happel87} \cite{Keller07a}). For a Krull--Schmidt category
$\cc$, let us denote by $\ind(\cc)$ the full subcategory whose objects form
a set of representatives for the isomorphism classes of the indecomposable
objects of $\cc$. By combining Happel's theorem~\ref{thm:Happel}
with the equivalence~(\ref{eq:derived-eq-stable}), we obtain
an isomorphism
\begin{equation}
\label{eq:isom-kZQ-ind-stable}
\xymatrix{ 
k(\Z Q) \ar[r]^-{\sim} & \ind(\ul{\mod}(\cp)).
}
\end{equation}
Let $s_i$, $i\in Q_0$, denote the vertices of $\Z Q$ corresponding to the
simple $A$-modules $S_i$ (considered as $T(A)$-modules concentrated
in degree $0$). Let $h$ be the Coxeter number of $\Delta$. Let $C$ be the
set of the following vertices of $\Z Q$:
\begin{equation} \label{eq:vertices-in-C}
\tau^{p(h-1)} \tau^{-1}\Sigma^{-1} s_i \ko p\in\Z \ko i\in Q_0.
\end{equation}

\begin{proposition} \label{prop:mesh-category-isomorphisms}
The isomorphism~(\ref{eq:isom-kZQ-ind-stable}) lifts to an isomorphism
\begin{equation}
\xymatrix{
\cR_C \ar[r]^-\sim & \ind(\mod(\cp))
}
\end{equation}
taking $s_i$ to the simple module $S_i$, $i\in Q_0$. It induces an
isomorphism
\begin{equation}
\xymatrix{
\cs_C \ar[r]^-\sim & \ind(\cp).
}
\end{equation}
\end{proposition} 

\begin{proof} The category $\cp$ is locally bounded and locally
representation-finite. Moreover, it is directed. It follows from 
\cite{BautistaGabrielRoiterSalmeron85} that $\cp$ is standard,
i.e. the category $\ind(\mod(\cp))$ is isomorphic to the mesh category
of the Auslander--Reiten quiver $\Gamma_{\mod(\cp)}$ (with the mesh relations
associated with the non projective vertices), cf.~\cite{Riedtmann85}.
It remains to be checked that the Auslander--Reiten quiver $\Gamma_{\mod(\cp)}$
is indeed obtained from $\Z Q$ by adding a new vertex $\sigma(x)$ and new 
arrows $\tau(x)\to\sigma(x)\to x$ to $\Z Q$ for each vertex $x$ in $C$.  
Indeed, the quiver $\Z Q$ is isomorphic to the stable Auslander--Reiten quiver 
$\Gamma_{\ul{\mod}(\cp)}$ via the isomorphism induced 
by~(\ref{eq:isom-kZQ-ind-stable}),  and we know that we have to insert the
vertex $v(P)$ corresponding to an indecomposable projective $P$
in the mesh starting at $v(\rad(P))$. If $P_M$ is the projective cover
of a simple module $M$, we have the exact sequence
\[
0 \to \rad(P_M) \to P_M \to M \to 0 \ko
\]
which shows that $\rad(P_M)=\Sigma^{-1} M$ in the triangulated
category $\ul{\mod}(\cp)$. Thus, we have to insert $v(P_M)$ in the
mesh {\em ending} at $\tau^{-1} \Sigma^{-1} v(M)$. Now the simple
$\cp$-modules are of the form $S_i\langle p\rangle$, $i\in Q_0$,
$p\in \Z$. It is well-known and not hard to check that the shift
$\langle 1 \rangle$ induces the composition $S \Sigma$ in the
stable category $\ul{\mod}{(\cp)}$ (equivalent to $\cd_Q$), where
$S$ is the Serre functor. Now we have
\[
S \Sigma = \tau \Sigma \Sigma = \tau \Sigma^2 = \tau \tau^{-h} = \tau^{-(h-1)}.
\]
Here, the isomorphism $\Sigma^2=\tau^{-h}$ follows
from Happel's theorem~\ref{thm:Happel} and from
Gabriel's description of the Serre functor (alias Nakayama functor)
in Proposition~6.5 of \cite{Gabriel80}. A detailed proof of a more
precise statement is given by Miyachi-Yekutieli in Theorem~4.1 of \cite{MiyachiYekutieli01}. Thus, we get 
$\tau^{-1}\Sigma^{-1} S_i\langle -p\rangle = \tau^{p(h-1)} \tau^{-1} \Sigma^{-1} S_i$, 
which proves the claim. This construction of $C$ also
makes the second assertion clear.
\end{proof}

Let $C\subset \Z Q_0$ be as in the Proposition. It satisfies 
Assumption~\ref{main-assumption} by the remark following
the statement of the assumption.
Let $M \in \mod(\cp)$ be a finite-dimensional module 
(i.e. a finite-dimensional module over the repetitive algebra).
The proposition shows that we may consider $M$ as
a module over the singular Nakajima category $\cs_C$,
and that we can consider its intermediate extension
\[
K_{LR}(M) \in \mod(\mod(\cp))
\]
as a module over the regular Nakajima category $\cR_C$.
By part c) of Lemma~\ref{lemma:Auslander-category}, 
the intermediate extension $K_{LR}(M)$
is rigid, since each projective in $\mod(\cp)$ is also injective.
Thus, from Theorem~\ref{thm:desingularization}, 
we obtain the following Corollary.

\begin{corollary} \label{cor:desing-quiver-grassmannian}
The map $\pi^{bs}$ of 
Theorem~\ref{thm:desingularization}
provides a desingularization of the quiver Grassmannian of $M$.
\end{corollary}

\subsection{Link to Cerulli--Feigin--Reineke's desingularization}
\label{ss:link-to-CFRs-desingularization}
In their article \cite{CerulliFeiginReineke12a},
Cerulli--Feigin--Reineke have constructed desingularizations of
quiver Grassmannians of representations of Dynkin quivers. 
We will show how their construction fits into the framework of 
desingularizations of quiver Grassmannians of modules over
repetitive algebras of 
section~\ref{ss:piecewise-hereditary}.

Let $Q$ be a connected Dynkin quiver and $A$ the path algebra of $Q$.
Following section~4 of \cite{CerulliFeiginReineke12a}, we define $\ch_Q$ 
to be the full subcategory of the category of morphisms of $\mod(A)$
whose objects are the injective morphisms
\[
f: P_1 \to P_0
\]
such that
\begin{itemize}
\item[1)] $P_0$ and $P_1$ are finitely generated projective $A$-modules and
\item[2)] $f$ does not admit a non zero direct factor of the form $0 \to P$, where
$P$ is a finitely generated projective $A$-module.
\end{itemize}
We define a commutative square of functors
\[
\xymatrix{ \ch_Q \ar[r] & \mod(\cp)=\grm(T(A)) \\
\proj(A) \ar[r] \ar[u] & \cp \ar[u]}
\]
as follows:
\begin{itemize}
\item[-] the functor $\proj(A) \to \ch_Q$ takes a module $P$ to the identity
$P \to P$, 
\item[-] the functor $\cp \to \mod(\cp)$ is the Yoneda embedding, 
\item[-] the functor $\ch_Q \to \grm(T(A))$ takes a morphism $f: P_1 \to P_0$ 
to the graded module $P_1 \leadsto \nu(P_0)$ which has $P_1$ in degree $0$,
$\nu(P_0)$ in degree $1$, $\nu(f): \nu(P_1) \to \nu(P_0)$ as the structural
morphism and all other components equal to zero, 
\item[-] the functor $\proj(A) \to \cp$
takes a module $P$ to the graded $T(A)$-module $(P \leadsto \nu (P))$,
where $P$ in degree $0$ is linked to $\nu (P)$ in degree $1$ by the
identity map $\nu(P) \to \nu(P)$ and all other components vanish. 
\end{itemize}
Notice that all four functors are fully faithful.
Using the horizontal functors, we can restrict a $\cp$-module to
$\proj(A)$ and a $\mod(\cp)$-module to $\ch_Q$. In this way,
the category $\mod(A)$ identifies with the subcategory of $\mod(\cp)$
formed by the modules supported on $\proj(A)$ and the category
$\mod(\ch_Q)$ with the full subcategory of $\mod(\mod(\cp))$ formed
by the modules supported on $\ch_Q$. 

Following section~5 of \cite{CerulliFeiginReineke12a}, for an $A$-module
$M$, we define the $\ch_Q$-module $\hat{M}$ by
\[
\hat{M}(P_1 \to P_0) = \im(\Hom(P_0, M) \to \Hom(P_1,M)).
\]
The following proposition shows that the functor 
$\Lambda: M \mapsto \hat{M}$ of [loc. cit.] is a particular
case of the intermediate extension $M \mapsto K_{LR}(M)$.
Therefore, Corollary~\ref{cor:desing-quiver-grassmannian} generalizes
Corollary~7.7 of \cite{CerulliFeiginReineke12a}.

\begin{proposition} 
Let $M$ be an $A$-module identified with a $\cp$-module
supported on $\proj(A) \subset \cp$. Then the $\mod(\cp)$-module $K_{LR}(M)$
is supported on $\ch_Q$ and its restriction to $\ch_Q$ is canonically
isomorphic to $\hat{M}$.
\end{proposition}

\begin{proof} Let $M^\wedge: \mod(\cp) \to \mod(k)$ be the functor
represented by $M$. We have $\res(M^\wedge)=M$. Thus, by part c) of
Lemma~5.4 of \cite{KellerScherotzke13a}, 
the module $K_{LR}(M)= K_{LR}(\res(M^\wedge))$
is the submodule of $M^\wedge$ generated by the images of all morphisms
$P^\wedge \to M^\wedge$, where $P$ belongs to $\cp$. Let us check that
the restriction of $K_{LR}(M)$ to $\ch_Q$ is isomorphic to $\hat{M}$.
Let
\[
\xymatrix{ 0 \ar[r] & P_1^M \ar[r] & P_0^M \ar[r] & M \ar[r] & 0}
\]
be a minimal projective resolution of $M$ in $\mod(A)$. We deduce that
we have a minimal projective resolution of $(M \leadsto 0)$ in $\mod(\cp)$ given by
\[
\xymatrix{
0 \ar[r] & (P_1^M \leadsto \nu(P_0^M)) \ar[r] & (P_0^M \leadsto \nu(P_0^M)) \ar[r] &
(M \leadsto 0) \ar[r] & 0.}
\]
For an arbitrary object $L$ of $\mod(\cp)$, a morphism $L \to (M \leadsto 0)$
factors through a projective if and only if it factors through 
$(P_0^M \leadsto \nu(P_0^M))$. Using this we see that for an object
$P_1 \to P_0$ of $\ch_Q$, the module 
\[
K_{LR}(P_1 \leadsto \nu(P_0)) 
\]
is the image
\[
 \im(\Hom((P_1 \leadsto \nu(P_0)), 
(P_0^M \leadsto \nu(P_0^M))) \to \Hom((P_1 \leadsto \nu(P_0)), (M \leadsto 0)).
\]
Clearly this image identifies with
\[
\im(\Hom(P_0,M) \to \Hom(P_1, M)) = \hat{M}(P_1 \to P_0).
\]
It remains to be shown that $K_{LR}(M)$ vanishes at all indecomposables
$L$ not belonging to the image of $\ch_Q$ in $\mod(\cp)$. Indeed, for
an object $L$ of $\mod(\cp)$, let $\Omega(L)$ denote the kernel of
a projective cover and $\Omega^{-1}(L)$ the cokernel of an
injective hull. Since $Q$ is a Dynkin quiver and the stable category of
$\mod(\cp)$ is equivalent to the derived category of $\mod(A)$, the
indecomposable objects of $\mod(\cp)$ are exactly the projective-injective
indecomposables and the objects $\Omega^p(L)$, where $p\in\Z$ and
$L$ is an indecomposable $A$-module. Clearly, the only indecomposable
projective objects possibly admitting a non zero morphism to $M \leadsto 0$
are the $P\leadsto \nu(P)$, where $P$ is an indecomposable projective
$A$-module. Now let $L$ be an indecomposable $A$-module
and
\[
\xymatrix{ 0 \ar[r] & P_1^L \ar[r] & P_0^L \ar[r] & L \ar[r] & 0}
\]
a minimal projective resolution. We have $K_{LR}(M)(L\leadsto 0) =0$ since
\[
\Hom((L\leadsto 0), (P_0^M \leadsto \nu(P_0^M))) =0.
\]
We have 
\[
\Omega(L) = (P_1^L \leadsto \nu(P_0^L))
\]
and so this object belongs to $\ch_Q$. It is easy to check that for $p\geq 2$,
the object $\Omega^p(L)$ has vanishing component in degree $0$ and
so does not admit non zero morphisms to $M$. Now let
\[
\xymatrix{ 0 \ar[r] & L \ar[r] & I^0_L \ar[r] & I^1_L \ar[r] & 0}
\]
be a minimal injective coresolution. We have
\[
\Omega^{-1}(L) = (\nu^{-1}(I^0_L) \leadsto I^1_L)
\]
and so $\Hom(\Omega^{-1}(L), (0 \leadsto M))=0$, where now the two components
are concentrated in degrees $-1$ and $0$. For $p\geq 2$, the object $\Omega^{-p}(L)$
has vanishing component in degree $0$ and so we have
$\Hom(\Omega^{-p}(L),(0 \leadsto M))=0$ for $p\geq 2$ as well.
\end{proof}

\subsection{The example $A_3$} \label{ss:example-A3}
We illustrate section~\ref{ss:link-to-CFRs-desingularization} by an example.
Let $Q$ be the linearly oriented quiver 
\[
\xymatrix{1 \ar[r] & 2 \ar[r] & 3}
\]
of type $A_3$. Figure~\ref{fig:example-A3} shows the quiver
of $\cR_C$, where $C$ is the configuration obtained from
Proposition~\ref{prop:mesh-category-isomorphisms}. 
Thus, the quiver is also the Auslander-Reiten quiver
of the repetitive algebra of $Q$. It contains the Auslander-Reiten
quiver of the path algebra $kQ$ as the triangle whose base
is formed by the simples $S_1$, $S_2$, $S_3$. The
vertices marked by $\bullet$ correspond to the 
indecomposables in the image of the functor
$\ch_Q \to \mod(\cp)$.

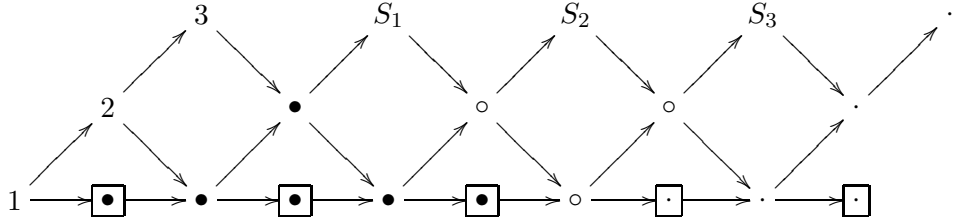
\begin{figure}
\[
\begin{xy} 0;<0.7pt,0pt>:<0pt,-0.7pt>:: 
(0,100) *+{1} ="0",
(50,50) *+{2} ="1",
(100,0) *+{3} ="2",
(50,100) *+[F]{\bt} ="3",
(100,100) *+{\bt} ="4",
(150,50) *+{\bt} ="5",
(200,0) *+{S_1} ="6",
(150,100) *+[F]{\bt} ="7",
(200,100) *+{\bt} ="8",
(250,50) *+{\circ} ="9",
(300,0) *+{S_2} ="10",
(250,100) *+[F]{\bt} ="11",
(300,100) *+{\circ} ="12",
(350,50) *+{\circ} ="13",
(400,0) *+{S_3} ="14",
(350,100) *+[F]{\cdot} ="15",
(400,100) *+{\cdot} ="16",
(450,50) *+{\cdot} ="17",
(500,0) *+{\cdot} ="18",
(450,100) *+[F]{\cdot} ="19",
"0", {\ar"1"},
"0", {\ar"3"},
"1", {\ar"2"},
"1", {\ar"4"},
"2", {\ar"5"},
"3", {\ar"4"},
"4", {\ar"5"},
"4", {\ar"7"},
"5", {\ar"6"},
"5", {\ar"8"},
"6", {\ar"9"},
"7", {\ar"8"},
"8", {\ar"9"},
"8", {\ar"11"},
"9", {\ar"10"},
"9", {\ar"12"},
"10", {\ar"13"},
"11", {\ar"12"},
"12", {\ar"13"},
"12", {\ar"15"},
"13", {\ar"14"},
"13", {\ar"16"},
"14", {\ar"17"},
"15", {\ar"16"},
"16", {\ar"17"},
"16", {\ar"19"},
"17", {\ar"18"},
\end{xy}
\]
\caption{Example $A_3$}
\label{fig:example-A3}
\end{figure}

\section{An example in tilted type $D_4$}
\label{s:example-tilted-D4}

We illustrate Corollary~\ref{cor:desing-quiver-grassmannian}
with an example of tilted type $D_4$.
We consider the algebra $B$ given by the square
\[
\xymatrix{
1 \ar[d] \ar[r] & 2 \ar[d] \\
3 \ar[r] & 4
}
\]
with the commutativity relation ($B$ is tilted of type $D_4$). 
Let $M$ be the $B$-module given  as the direct sum of the three modules $I_1=P_4$, $P_2$ and $I_3$:
\[
\xymatrix{
k   & k \ar[l]_{\id} \\
k \ar[u]^{\id} & k \ar[l]^{\id} \ar[u]_{\id}
}
\quad
\xymatrix{
k   & k \ar[l]_{\id} \\
0 \ar[u] & 0 \ar[l] \ar[u]
}
\quad
\xymatrix{
0  & 0 \ar[l] \\
k \ar[u] & k. \ar[l]_{\id} \ar[u]
}
\]
All submodules of $M$ with dimension vector $(1,1,1,1)$ are isomorphic to 
one of the modules  $I_1$ and  $P_2\oplus I_3$.
The space $\Hom( I_1, P_2 \oplus I_3 )$ is of dimension one.
Let us denote by $L$ the submodule isomorphic to $P_2 \oplus I_3$,
where we embed the factor $P_2$ into $P_4$. We obtain 
$\Hom(L, M/L)\cong \Hom(P_2\oplus I_3, P_2\oplus I_3)$, 
which is two-dimensional. 
Let $N$ denote the submodule isomorphic to $P_2 \oplus I_3$, where we embed $N$
into $P_2\oplus I_3$. In this case, we have
$\Hom(N,  M/N) \cong \Hom(P_2\oplus I_3, I_1)$,
which is one-dimensional.  
Thus, the 
tangent spaces of the quiver Grassmannian $\Gr_{(1,1,1,1)^t}(M)$ 
at the points $L$ and $N$ do not have the same dimension.  
In fact the quiver Grassmannian consists of two irreducible components, 
both isomorphic to the projective line, and which intersect at the singular point $L$. 
The two components are given by the closures of $S_{[I_1]}$ and $S_{[P_2 \oplus I_3]}$, where 
$S_{[N]} $ denotes the irreducible and locally closed subset in $\Gr_{(1,1,1,1)^t}(M)$ of submodules isomorphic to $N$. 
Hence $\{v_1, v_2 \}=\cv_w(M)$ is given by $\dimv K_{LR} (I_1)=(v_1,w)$ and $\dimv K_{LR}(P_2 \oplus I_3)=(v_2,w)$. 

Let us choose the orientation
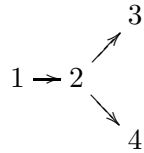
\[
\xymatrix@R=0.35cm@C=0.35cm{
& & 3\\
1 \ar[r] & 2 \ar[ru] \ar[rd] &\\
&& 4
}
\]
of the Dynkin diagram $D_4$. Removing the boxes from the quiver 
in Figure~\ref{fig:AR-quiver-rep-alg-D4} gives the Auslander-Reiten 
quiver of the derived category $\cd_Q$.
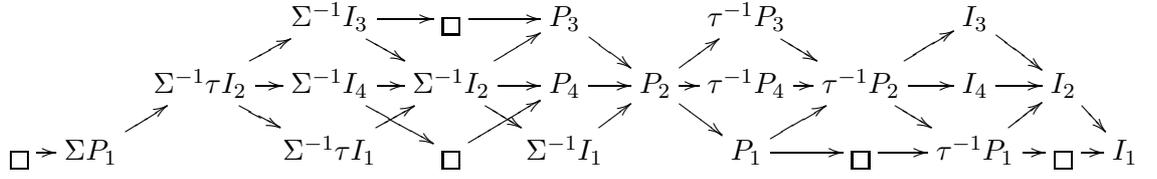
\begin{figure}
\[
\xymatrix@R=0.25cm@C=0.25cm{ 
&&& \Sigma^{-1} I_3  \ar[rd] \ar[r] & \boxed{} \ar[r]& P_3 \ar[rd] & & \tau^{-1} P_3 \ar[rd] & & I_3 \ar[rd] & &  \\
&& \Sigma^{-1} \tau I_2\ar[ru] \ar[rd] \ar[r] & \Sigma^{-1} I_4 \ar[r] \ar[rd] & \Sigma^{-1} I_2 \ar[r] \ar[rd] \ar[ru] & P_4 \ar[r] & P_2\ar[r] \ar[rd] \ar[ru]& \tau^{-1} P_4 \ar[r] & \tau^{-1} P_2\ar[r] \ar[rd] \ar[ru] & I_4 \ar[r] & I_2 \ar[rd]  &      \\
\boxed{}  \ar[r] & \Sigma P_1   \ar[ru] && \Sigma^{-1} \tau I_1 \ar[ru] & \boxed{} \ar[ru] & \Sigma^{-1} I_1 \ar[ru] & &  P_1 \ar[ru] \ar[r] & \boxed{} \ar[r] & \tau^{-1} P_1 \ar[ru] \ar[r] 
& \boxed{} \ar[r] & I_1 
}
\]
\caption{The AR-quiver of the repetitive algebra of $D_4$}
\label{fig:AR-quiver-rep-alg-D4}
\end{figure}
In this category, the suspension functor $\Sigma$ is isomorphic
to $\tau^{-3}$ and the Coxeter number of $D_4$ equals $h=6$.
Thus, the configuration $C$ of section~\ref{ss:piecewise-hereditary}
 is given by the vertices
corresponding to the objects $\tau^{5p+2} S_i$ of $\cd_Q$ and
the quiver of $\cR_C$ is the one displayed in Figure~\ref{fig:AR-quiver-rep-alg-D4}.
Using Theorem~2.4 of \cite{KellerScherotzke13a} one verifies that
the quiver with relations of $\cs_C$ is the quiver
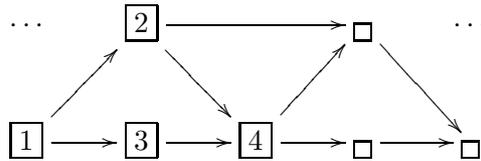
\[
\xymatrix{ 
\cdots  & \boxed{2} \ar[rd] \ar[rr] & &\boxed{} \ar[rd] & \cdots \\
\boxed{1} \ar[ru] \ar[r] & \boxed{3} \ar[r] & \boxed{4} \ar[ru] \ar[r]& \boxed{} \ar[r] & \boxed{}
}
\]
with commuting relations in the squares and no relations in the triangles. 
The intermediate extension of $M$ is the direct sum of  the intermediate extensions 
of its three summands, namely the modules
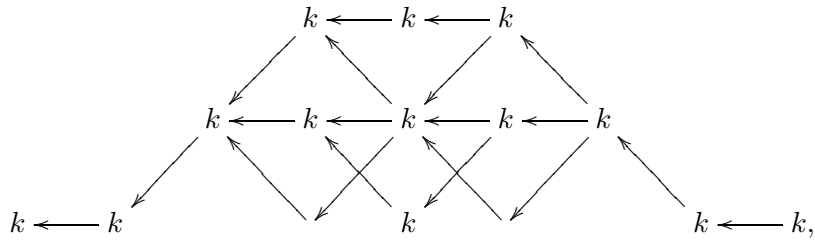
\[
\xymatrix{ 
&&&  k \ar[ld] & k \ar[l]& k \ar[ld] \ar[l]& &   & \\
&& k \ar[ld]  &k \ar[l]  & k \ar[l] \ar[lu] \ar[ld] & k \ar[ld] \ar[l] &  k \ar[l] \ar[lu] \ar[ld]&  & \\
k &  k \ar[l] && \ar[lu] & k \ar[lu] & \ar[lu] & & k  \ar[lu]  & k, \ar[l]
}
\]
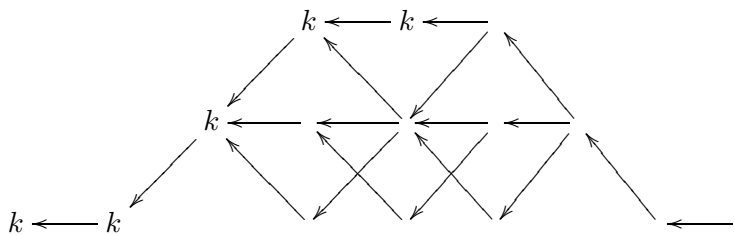
\[
\xymatrix{ 
&&&  k \ar[ld] & k \ar[l]&  \ar[ld] \ar[l]& &   & \\
&& k\ar[ld]  &  \ar[l]  &\ar[l] \ar[lu] \ar[ld] &  \ar[ld] \ar[l] &  \ar[l] \ar[lu] \ar[ld]&  & \\
k & k  \ar[l] && \ar[lu] &   \ar[lu] & \ar[lu] & &   \ar[lu]  &  \ar[l]
}
\]
and 
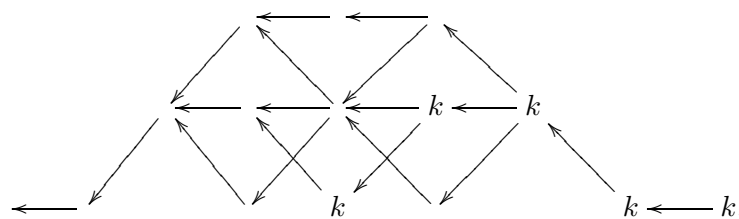
\[
\xymatrix{ 
&&&   \ar[ld] &  \ar[l]&  \ar[ld] \ar[l]  & &   & \\
&& \ar[ld]  & \ar[l]  &\ar[l] \ar[lu] \ar[ld] & k \ar[ld] \ar[l] & k  \ar[l] \ar[lu] \ar[ld]&  & \\
 &   \ar[l] && \ar[lu] & k \ar[lu] & \ar[lu] & &  k \ar[lu]  & k. \ar[l]
}
\]
Here all non-labelled vertices are represented by zero spaces and all
possibly non zero maps are the identity.
The intermediate extension of the generic subrepresentations is given by the first summand $K_{LR} ( I_1)$ and 
the sum of the last two summands $K_{LR} (P_2) \oplus K_{LR} (I_3)$ of $K_{LR} (M)$. 
We conclude that the desingularization map of 
Corollary~\ref{cor:desing-quiver-grassmannian} is the map
\[
\Gr_{\dimv K_{LR} (I_1)}(K_{LR} (M) ) \coprod  \Gr_{\dimv K_{LR} (I_3\oplus P_2) }(K_{LR} (M) ) 
\to \Gr_{(1,1,1,1)^t}(M)
\]
taking $U$ to $\res(U)$.
Finally, it is easy to see that both Grassmannians on the left hand side 
are isomorphic to projective lines.


\begin{thebibliography}{10}

\bibitem{BautistaGabrielRoiterSalmeron85}
Raymundo Bautista, Peter Gabriel, Andrei~V. Roiter, and Luis Salmer\'on,
  \emph{Representation-finite algebras and multiplicative bases [provisional]},
  Inv. math. \textbf{81} (1985), 271--285.

\bibitem{CalderoChapoton06}
Philippe Caldero and Fr{\'e}d{\'e}ric Chapoton, \emph{Cluster algebras as
  {H}all algebras of quiver representations}, Comment. Math. Helv. \textbf{81}
  (2006), no.~3, 595--616.

\bibitem{CalderoKeller08}
Philippe Caldero and Bernhard Keller, \emph{From triangulated categories to
  cluster algebras}, Inv. Math. \textbf{172} (2008), 169--211.

\bibitem{CerulliFeiginReineke12b}
Giovanni Cerulli~Irelli, Evgeny Feigin, and Markus Reineke, \emph{Degenerate
  flag varieties: moment graphs and {S}chr\"oder numbers}, preprint,
  arXiv:1206.4178 [math.AG], to appear in J. Alg. Comb.

\bibitem{CerulliFeiginReineke12a}
\bysame, \emph{Desingularization of quiver {G}rassmannians for {D}ynkin
  quivers}, preprint, arXiv:1209.3960 [math.AG]. To appear in Advances.

\bibitem{CerulliFeiginReineke13}
\bysame, \emph{Homological approach to the {H}ernandez--{L}eclerc construction
  and quiver varieties}, preprint, arXiv:1302.5297 [math.AG].

\bibitem{CerulliFeiginReineke12}
\bysame, \emph{Quiver grassmannians and degenerate flag varieties}, Algebra and
  Number Theory \textbf{6} (2012), 165--194.

\bibitem{DerksenWeymanZelevinsky10}
Harm Derksen, Jerzy Weyman, and Andrei Zelevinsky, \emph{Quivers with
  potentials and their representations {II}: {Applications to cluster
  algebras}}, J.~Amer.~Math.~Soc. \textbf{23} (2010), 749--790.

\bibitem{Efimov11}
Alexander Efimov, \emph{Quantum cluster variables via vanishing cycles},
  arXiv:1112.3601 [math.AG].

\bibitem{Feigin11}
Evgeny Feigin, \emph{Degenerate flag varieties and the median {G}enocchi
  numbers}, Math. Res. Lett. \textbf{18} (2011), no.~6, 1163--1178.
  \MR{2915473}

\bibitem{Feigin12}
\bysame, \emph{{$\Bbb{G}_a^M$} degeneration of flag varieties}, Selecta Math.
  (N.S.) \textbf{18} (2012), no.~3, 513--537. \MR{2960025}

\bibitem{FeiginFinkelberg11}
Evgeny Feigin and Michael Finkelberg, \emph{Degenerate flag varieties of type
  $a$: {F}robenius splitting and {$BW$} theorem}, arXiv:1103.1491 [math.AG].

\bibitem{FominZelevinsky02}
Sergey Fomin and Andrei Zelevinsky, \emph{Cluster algebras. {I}.
  {F}oundations}, J. Amer. Math. Soc. \textbf{15} (2002), no.~2, 497--529
  (electronic).

\bibitem{Gabriel62}
Peter Gabriel, \emph{Des cat{\'e}gories ab{\'e}liennes}, Bull. Soc. Math.
  France \textbf{90} (1962), 323--448.

\bibitem{Gabriel80}
Peter Gabriel, \emph{Auslander-{R}eiten sequences and representation-finite
  algebras}, Representation theory, I (Proc. Workshop, Carleton Univ., Ottawa,
  Ont., 1979), Springer, Berlin, 1980, pp.~1--71.

\bibitem{Happel87}
Dieter Happel, \emph{On the derived category of a finite-dimensional algebra},
  Comment. Math. Helv. \textbf{62} (1987), no.~3, 339--389.

\bibitem{HernandezLeclerc10}
David Hernandez and Bernard Leclerc, \emph{Cluster algebras and quantum affine
  algebras}, Duke Math. J. \textbf{154} (2010), no.~2, 265--341.

\bibitem{Hille96}
Lutz Hille, \emph{Tilting line bundles and moduli of thin sincere
  representations of quivers}, An. \c Stiin\c t. Univ. Ovidius Constan\c ta
  Ser. Mat. \textbf{4} (1996), no.~2, 76--82, Representation theory of groups,
  algebras, and orders (Constan{\c{t}}a, 1995).

\bibitem{Keller07a}
Bernhard Keller, \emph{Derived categories and tilting}, Handbook of Tilting
  Theory, LMS Lecture Note Series, vol. 332, Cambridge Univ. Press, Cambridge,
  2007, pp.~49--104.

\bibitem{KellerScherotzke13a}
Bernhard Keller and Sarah Scherotzke, \emph{Graded quiver varieties and derived
  categories}, arXiv:1303.2318 [math.RT].

\bibitem{KimuraQin12}
Yoshiyuki Kimura and Fan Qin, \emph{Quiver varieties and quantum cluster
  algebras}, arXiv:1205.2066 [math.RT].

\bibitem{LebruynProcesi90}
Lieven Le~Bruyn and Claudio Procesi, \emph{Semisimple representations of
  quivers}, Trans. Amer. Math. Soc. \textbf{317} (1990), no.~2, 585--598.

\bibitem{LeclercPlamondon12}
Bernard Leclerc and Pierre-Guy Plamondon, \emph{Nakajima varieties and
  repetitive algebras}, preprint, arXiv:1208.3910 [math.QA]. To appear in Publ.
  RIMS, Kyoto.

\bibitem{Lusztig98a}
G.~Lusztig, \emph{On quiver varieties}, Adv. Math. \textbf{136} (1998), no.~1,
  141--182.

\bibitem{MiyachiYekutieli01}
Jun-ichi Miyachi and Amnon Yekutieli, \emph{Derived {P}icard groups of
  finite-dimensional hereditary algebras}, Compositio Math. \textbf{129}
  (2001), no.~3, 341--368.

\bibitem{Nakajima01}
Hiraku Nakajima, \emph{Quiver varieties and finite-dimensional representations
  of quantum affine algebras}, J. Amer. Math. Soc. \textbf{14} (2001), no.~1,
  145--238 (electronic).

\bibitem{Nakajima11}
\bysame, \emph{Quiver varieties and cluster algebras}, Kyoto Journal of
  Mathematics \textbf{51} (2011), no.~1, 71--126.

\bibitem{Qin12}
Fan Qin, \emph{Alg\`ebres amass\'ees quantiques acycliques}, Ph.~D.~Thesis,
  Universit\'e Paris Diderot -- Paris 7, May 2012.

\bibitem{Qin12a}
\bysame, \emph{$t$-analogue of $q$-characters, bases of quantum cluster
  algebras, and a correction technique}, arXiv:1207.6604 [math.QA]. To appear
  in IMRN.

\bibitem{Reineke12}
Markus Reineke, \emph{Every projective variety is a quiver {G}rassmannian},
  math.RT/1204.5730. To appear in Algebras and Representation Theory.

\bibitem{ReitenVanDenBergh02}
Idun Reiten and Michel Van~den Bergh, \emph{Noetherian hereditary abelian
  categories satisfying {S}erre duality}, J. Amer. Math. Soc. \textbf{15}
  (2002), no.~2, 295--366 (electronic).

\bibitem{Riedtmann80}
Christine Riedtmann, \emph{Algebren, {D}arstellungsk{\"o}cher,
  \"{U}berlagerungen und zur{\"u}ck}, Comment. Math. Helv. \textbf{55} (1980),
  no.~2, 199--224.

\bibitem{Riedtmann80a}
\bysame, \emph{Representation-finite self-injective algebras of class
  {$A\sb{n}$}}, Representation theory, II (Proc. Second Internat. Conf.,
  Carleton Univ., Ottawa, Ont., 1979), Lecture Notes in Math., vol. 832,
  Springer, Berlin, 1980, pp.~449--520.

\bibitem{Riedtmann85}
\bysame, \emph{Alg\`ebres de type de repr\'esentation fini (d'apr\`es
  {B}autista, {B}ongartz, {G}abriel, {R}o\u\i ter et d'autres)}, Ast\'erisque
  (1986), no.~133-134, 335--350, Seminar Bourbaki, Vol. 1984/85. \MR{837229
  (87g:16032)}

\bibitem{Ringel13}
Claus~Michael Ringel, \emph{Quiver {G}rassmannians and {A}uslander varieties
  for wild algebras}, arXiv:1305.4003 [math.RT].

\bibitem{Ringel84}
\bysame, \emph{Tame algebras and integral quadratic forms}, Lecture Notes in
  Mathematics, vol. 1099, Springer Verlag, 1984.

\bibitem{Schofield92a}
Aidan Schofield, \emph{General representations of quivers}, Proc. London Math.
  Soc. (3) \textbf{65} (1992), no.~1, 46--64. \MR{1162487 (93d:16014)}

\end{thebibliography}

\def\cprime{$'$} \def\cprime{$'$}
\providecommand{\bysame}{\leavevmode\hbox to3em{\hrulefill}\thinspace}
\providecommand{\MR}{\relax\ifhmode\unskip\space\fi MR }
\providecommand{\MRhref}[2]{%
  \href{http://www.ams.org/mathscinet-getitem?mr=#1}{#2}
}
\providecommand{\href}[2]{#2}

\end{document}